\documentclass[12pt]{amsart}
\usepackage{amsmath}
\usepackage[utf8]{inputenc}
\usepackage[usenames, dvipsnames]{color}
\usepackage{hyperref}
\usepackage{multicol}
\usepackage{amscd}
\usepackage{faktor}
\usepackage{bbm}
\usepackage{pb-diagram}
\usepackage{amssymb}
\usepackage{tikz-cd}
\usepackage{braket}
\usepackage{amsfonts}
\usepackage{caption}
\usepackage{amsthm}
\usepackage{enumitem}
\usepackage[english]{babel}
\usepackage[margin = 1in]{geometry}

\usepackage{soul}
\input xy
\usepackage[all]{xy}
\usepackage{tikz-cd}

\usepackage{placeins}

\newcommand{\T}{{\mathcal T}}
\newcommand{\cC}{{\mathcal C}}
\newcommand{\cR}{{\mathcal R}}

\newcommand{\cL}{{\mathcal L}}

\newcommand{\G}{{\mathcal G}}
\newcommand{\Z}{{\mathbb Z}}

\newcommand{\cH}{\mathcal{H}}

\newcommand{\cP}{\mathcal{P}}

\newcommand{\Id}{\operatorname{Id}}

\theoremstyle{plain}
\numberwithin{equation}{section}

\newtheorem{theorem}{Theorem}[section]
\newtheorem{lemma}[theorem]{Lemma}

\newtheorem{proposition}[theorem]{Proposition}
\newtheorem{algoritmo}[theorem]{}

\theoremstyle{definition}
\newtheorem{definition}[theorem]{Definition}
\newtheorem{example}[theorem]{Example}
\theoremstyle{remark}
\newtheorem{remark}[theorem]{Remark}
\newtheorem{notacion}[theorem]{Notation}
\numberwithin{figure}{section}

\definecolor{mycommentcolor}{RGB}{255,0,0} 

\theoremstyle{remark}

\newcounter{commentcounter}

\newcounter{todocounter}

\author[S. Cui]{Shawn X. Cui}
\address{ Department of Mathematics, Department of Physics and Astronomy, Purdue University, West Lafayette, IN, United States}
\email{cui177@purdue.edu}

\author[C. Galindo]{C\'esar Galindo}
\address{ Departamento de Matem\'aticas, Universidad de los Andes, Bogot\'a, Colombia}
\email{cn.galindo1116@uniandes.edu.co}

\author[D. Romero]{Diego Romero}
\address{ Departamento de Matem\'aticas, Universidad de los Andes, Bogot\'a, Colombia}
\email{da.romero12@uniandes.edu.co}

\begin{document}

\title[Abelian Group Quantum Error Correction in Kitaev's Model]{Quantum error correction in Kitaev's quantum double model for Abelian groups}

\thanks{ C.G. was partially supported by Grant INV-2023-162-2830 from the School of Science of Universidad de los Andes. D.R. was partially supported by the Grants INV-2023-164-2722 and INV-2022-150-2587 from the School of Science of Universidad de los Andes.  D.R. was partially supported by the startup grant of S.C. at Purdue. S.C. is partly supported by NSF grant CCF-2006667, ORNL-led Quantum Science Center, and ARO MURI. 
\newline
D.R. would like to thank the hospitality and excellent working conditions of the Department of Mathematics at Purdue University, where he carried out this research as a Visiting Scholar.  C.G. and D.R. thank Tristram Bogart for useful discussions.}

\begin{abstract}
In this paper, we present a detailed mathematical description of the error correction process for Kitaev's model for finite Abelian groups. The number of errors Kitaev's model can correct depends on the lattice and its topology. Although there is a theoretical maximum number of errors that can be corrected, we prove that correcting this number of errors, in general, is an NP-complete problem. Consequently, we introduce a polynomial-time correction algorithm that corrects a number of errors below the theoretical maximum.

\end{abstract}

\date{\today}

\maketitle
\section{Introduction}\label{intro}

Quantum computing has been an intense area of study for the past decades, with a key challenge being the manipulation of information while minimizing decoherence. This challenge fostered the development of quantum error correcting codes (QECC), which have evolved in significance and variety since their inception \cite{shor1995scheme}. Depending on the desired parameters such as encoding rate, distance, and threshold, different QECCs have been developed with the majority of them described by the stabilizer formalism \cite{Gottesman}. An exhaustive list of such codes is beyond the scope of the paper, but among some notable ones are, for example, CSS codes \cite{calderbank1996good, steane1996simple}, triorthogonal codes \cite{bravyi2012magic}, fiber bundle codes \cite{hastings2021fiber}, and lifted product codes \cite{panteleev2021quantum}.  These QECCs   involve the encoding of quantum states into qubits (i.e., two-state quantum systems), but similar concepts can also be defined on \emph{qudits}, higher-dimensional quantum systems, leading to a variety of coding structures \cite{gottesman1998fault}.

A large class of QECCs arise from exactly solvable lattice models which has topological quantum order. Prominent examples include the the Levin-Wen  path-net model \cite{Levin}, Kitaev's quantum double model \cite{anyons}, Haah code \cite{haah2011local}, and X-cube model \cite{vijay2015new}. In such models, the Hamiltonian is typically written as a sum of local commuting projectors, and the code subspace is the common eigenspace of the projectors. This generalizes the (Pauli) stabilizer formalism. Among the most well-known topological order is the toric/surface code, which involves arranging a 2D array of qubits on a toric surface and  has undergone extensive study \cite{lattice, calfault2}. The error correction process in toric code is well-established (e.g. \cite{calfault1}), with a threshold as high as $1\%$ in certain error models.


A natural question is how the QECC obtained from other topological order is compared to toric code. The latter is a special example of the Kitaev's quantum double model based on a finite group. When the group is chosen to be $\mathbb{Z}_2$, the model recovers the toric code. More generally, Abelian (resp. non-Abelian) groups correspond to Abelian (resp. non-Abelian) topological order. For the purpose of universal topological quantum computing, the use of non-Abelian topological order is necessary. Hence it is especially interesting to study the error correction in non-Abelian models. However, this seems challenging due to the complexity of the Hamiltonian and the quasi-particle creation/annihilation operators. Some attempts in this direction include error correction in the Fibonacci  path-net model \cite{schotte2022quantum, schotte2022fault} and the quantum double model based on the group $S_3$ \cite{wootton2014error}, though in the latter reference only error corresponding to charge exicitations were considered.

As a first step, in this paper we consider quantum error correction in Kitaev's model based on any Abelian group $G$. This model features an oriented lattice on a compact oriented surface $\Sigma$, with each edge associated with the space $\mathbb{C}[G]$, the group algebra of $G$. The Hamiltonian is of the form
\begin{align*}
    H = -\sum_{v} A(v) - \sum_{p} B(p),
\end{align*}
where $A(v)$ is a project associated with each vertex $v$ and $B(p)$ is a projector associated with each face $p$ (see Section \ref{kitae}). The code subspace is the common $+1$-eigenspace of the $A(v)'$s and $B(p)'$s.  Unlike the toric code, the description for error detection and correction in Abelian Kitaev model is not thoroughly understood.  While error correction for cyclic groups has been previously discussed \cite{clu3,clu4}, a similar study for general Abelian groups is still lacking. 
This article aims to address this gap by studying error correction in the case of general Abelian groups, improving the results in \cite{clu3,clu4}.


In Theorem \ref{teoricamente}, we establish that to correct errors affecting $n$ edges in Kitaev's Abelian model, each non-contractible cycle in $\Sigma$ must contain more than $2n$ edges. This is extends a well-known result for the toric code. However, Theorem \ref{propNP} reveals that for non-elementary abelian 2-groups, constructing correction operators as described in Theorem \ref{teoricamente} is an NP-complete problem. This complexity stems from the fact that the problem is reducible to the Minimum Tree Partition Problem (MTP), which is NP-complete. In contrast, the toric code has an efficient polynomial-time decoding algorithm.

A main goal of this article is to present Algorithm \ref{algoritm}, which provides a polynomial-time error correction method for Kitaev's model, albeit with a compromise on the number of correctable errors. For effective error correction, Algorithm \ref{algoritm} asks that every non-contractible cycle contains more than $f(n)=\left\lfloor n\left(\frac{2+ \log_2(n)}{2}\right) +1\right\rfloor$ edges. Notice that  $2n < f(n)$ and the function $f(n)$ is $ O(n\log_2(n))$, which is significant given that the order of growth for the optimal solution is $O(n)$ (see Section \ref{seccioncorrecion}).


The paper is organized as follows. Section \ref{kitae} provides a brief review of Kitaev's model for arbitrary groups. Section \ref{acase} focuses on a specific case of Kitaev's model, considering only Abelian groups. Section \ref{secciondeteccion} discuss how to perform error detection and present the process of syndrome extraction. Section \ref{syndrome} examines the information that can be derived about an error from its corresponding syndrome.
Section \ref{numeroteorico} identifies a specific set of correctable errors for Kitaev's model, establishing a theoretical number of errors that can be corrected. Section \ref{seccioncorrecion} develops Algorithm \ref{algoritm}, an error correction algorithm, which, to correct $n$ errors, must ensure that a non-contractible cycle cannot be formed in $\Sigma$ with $f(n)$ edges. Finally, Section \ref{apendice} shows that correcting the theoretical number of errors proposed in Section \ref{numeroteorico} is, in practice, an NP-complete problem.
\section{Kitaev's quantum double model}\label{kitae}

Consider a finite group $G$. The Kitaev's quantum double model based on $G$ is defined on a closed oriented surface $\Sigma$ together with a lattice $\mathcal{L} $ on it. The lattice $\mathcal{L}$ is constructed from vertices, edges, and faces (also called  plaquettes) where each edge is endowed with an arbitrary orientation. Each edge \( e \) is bounded by exactly two vertices \( v_1 \) and \( v_2 \); we say that \( v_1 \) and $v_2$ are incident to $e$, and that \( v_1 \) and \( v_2 \) are adjacent. Every edge belongs to the boundary of  exactly two faces. An  example of a generic lattice can see in Figure \ref{lattice}.

\begin{figure}[h!]
\centering
\tikzset{every picture/.style={line width=0.75pt}} 
\begin{tikzpicture}[x=0.65pt,y=0.65pt,yscale=-1,xscale=1]
\draw   (26,202) -- (42.5,147) -- (93.5,147) -- (110,202) -- cycle ;
\draw    (110,202) -- (236,202) ;
\draw    (93.5,147) -- (166,202) ;
\draw    (93.5,147) -- (177,147) ;
\draw    (177,147) -- (236,202) ;
\draw    (247,129) -- (236,202) ;
\draw    (176,98) -- (177,147) ;
\draw    (231,84) -- (176,98) ;
\draw    (247,129) -- (231,84) ;
\draw    (93.5,147) -- (121,98) ;
\draw    (121,98) -- (176,98) ;
\draw    (42.5,147) -- (30,103) ;
\draw    (30,103) -- (76,78) ;
\draw    (76,78) -- (121,98) ;
\draw    (76,78) -- (125,60) ;
\draw    (125,60) -- (176,98) ;
\draw    (125,60) -- (212,57) ;
\draw    (231,84) -- (212,57) ;
\draw    (76,78) -- (69,23) ;
\draw    (69,23) -- (133,23) ;
\draw    (133,23) -- (197,24) ;
\draw    (133,23) -- (125,60) ;
\draw    (197,24) -- (212,57) ;
\draw    (69,23) -- (30,103) ;
\draw  [fill={rgb, 255:red, 0; green, 0; blue, 0 }  ,fill opacity=1 ] (37.53,128.79) -- (30.71,122.65) -- (39.23,119.77) -- cycle ;
\draw  [fill={rgb, 255:red, 0; green, 0; blue, 0 }  ,fill opacity=1 ] (207.3,89.76) -- (201.09,96.52) -- (198.3,87.96) -- cycle ;
\draw  [fill={rgb, 255:red, 0; green, 0; blue, 0 }  ,fill opacity=1 ] (192.25,201.06) -- (200.18,196.44) -- (200.32,205.44) -- cycle ;
\draw  [fill={rgb, 255:red, 0; green, 0; blue, 0 }  ,fill opacity=1 ] (138.25,201.95) -- (130.31,206.55) -- (130.19,197.55) -- cycle ;
\draw  [fill={rgb, 255:red, 0; green, 0; blue, 0 }  ,fill opacity=1 ] (69.25,202.02) -- (61.23,206.48) -- (61.27,197.48) -- cycle ;
\draw  [fill={rgb, 255:red, 0; green, 0; blue, 0 }  ,fill opacity=1 ] (51.29,59.42) -- (51.73,68.59) -- (43.68,64.56) -- cycle ;
\draw  [fill={rgb, 255:red, 0; green, 0; blue, 0 }  ,fill opacity=1 ] (128.16,45.41) -- (125.43,36.65) -- (134.24,38.53) -- cycle ;
\draw  [fill={rgb, 255:red, 0; green, 0; blue, 0 }  ,fill opacity=1 ] (33.16,177.85) -- (31.01,168.92) -- (39.67,171.38) -- cycle ;
\draw  [fill={rgb, 255:red, 0; green, 0; blue, 0 }  ,fill opacity=1 ] (241.71,165.68) -- (238.26,157.18) -- (247.19,158.31) -- cycle ;
\draw  [fill={rgb, 255:red, 0; green, 0; blue, 0 }  ,fill opacity=1 ] (203.75,171.6) -- (212.52,174.3) -- (205.99,180.5) -- cycle ;
\draw  [fill={rgb, 255:red, 0; green, 0; blue, 0 }  ,fill opacity=1 ] (100.5,69) -- (105.8,61.5) -- (109.66,69.63) -- cycle ;
\draw  [fill={rgb, 255:red, 0; green, 0; blue, 0 }  ,fill opacity=1 ] (144.79,147.27) -- (136.88,151.92) -- (136.71,142.93) -- cycle ;
\draw  [fill={rgb, 255:red, 0; green, 0; blue, 0 }  ,fill opacity=1 ] (152.5,98.02) -- (144.48,102.48) -- (144.52,93.48) -- cycle ;
\draw  [fill={rgb, 255:red, 0; green, 0; blue, 0 }  ,fill opacity=1 ] (56.53,88.63) -- (51.57,96.35) -- (47.36,88.4) -- cycle ;
\draw  [fill={rgb, 255:red, 0; green, 0; blue, 0 }  ,fill opacity=1 ] (105.44,126.07) -- (105.04,116.9) -- (113.07,120.97) -- cycle ;
\draw  [fill={rgb, 255:red, 0; green, 0; blue, 0 }  ,fill opacity=1 ] (94.89,86.28) -- (104.05,85.66) -- (100.18,93.78) -- cycle ;
\draw  [fill={rgb, 255:red, 0; green, 0; blue, 0 }  ,fill opacity=1 ] (104.53,183.79) -- (97.71,177.65) -- (106.23,174.77) -- cycle ;
\draw  [fill={rgb, 255:red, 0; green, 0; blue, 0 }  ,fill opacity=1 ] (153.65,80.47) -- (144.58,79.08) -- (150.13,71.99) -- cycle ;
\draw  [fill={rgb, 255:red, 0; green, 0; blue, 0 }  ,fill opacity=1 ] (73.25,147) -- (65.24,151.5) -- (65.26,142.5) -- cycle ;
\draw  [fill={rgb, 255:red, 0; green, 0; blue, 0 }  ,fill opacity=1 ] (130.21,174.3) -- (139.24,175.91) -- (133.52,182.86) -- cycle ;
\draw  [fill={rgb, 255:red, 0; green, 0; blue, 0 }  ,fill opacity=1 ] (223.76,73.8) -- (215.53,69.74) -- (222.95,64.66) -- cycle ;
\draw  [fill={rgb, 255:red, 0; green, 0; blue, 0 }  ,fill opacity=1 ] (240.54,110.19) -- (233.31,104.53) -- (241.62,101.08) -- cycle ;
\draw  [fill={rgb, 255:red, 0; green, 0; blue, 0 }  ,fill opacity=1 ] (206.3,44.07) -- (198.68,38.96) -- (206.72,34.9) -- cycle ;
\draw  [fill={rgb, 255:red, 0; green, 0; blue, 0 }  ,fill opacity=1 ] (169,23.67) -- (160.82,27.83) -- (161.19,18.84) -- cycle ;
\draw  [fill={rgb, 255:red, 0; green, 0; blue, 0 }  ,fill opacity=1 ] (106,23.03) -- (97.96,27.47) -- (98.04,18.47) -- cycle ;
\draw  [fill={rgb, 255:red, 0; green, 0; blue, 0 }  ,fill opacity=1 ] (176.55,129.5) -- (171.95,121.55) -- (180.95,121.45) -- cycle ;
\draw  [fill={rgb, 255:red, 0; green, 0; blue, 0 }  ,fill opacity=1 ] (178.04,57.85) -- (170.04,62.34) -- (170.05,53.34) -- cycle ;
\draw  [fill={rgb, 255:red, 0; green, 0; blue, 0 }  ,fill opacity=1 ] (72.58,51.87) -- (77.98,59.3) -- (69.04,60.34) -- cycle ;
\end{tikzpicture}
\caption{Example of a generic lattice $\cL$.}
\label{lattice}
\end{figure}
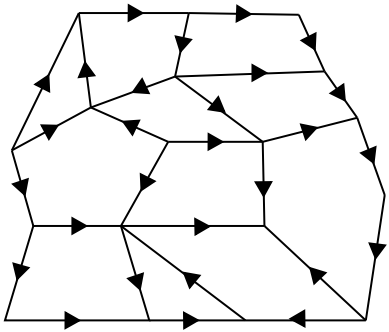

We adopt the following notation:
\begin{align*}
    V= \{\text{vertices of }\cL\} && E= \{\text{edges of }\cL\} && P= \{\text{faces of }\cL\}
\end{align*}
To each edge \(e \in E\), we associate the Hilbert space \(\mathcal{H}_e = \mathbb{C}[G] = \operatorname{span}\{ \ket{g} : g \in G\}\). The \emph{total space} associated with \(\mathcal{L}\) is defined as 
\[
\mathcal{H}_{\text{tot}} = \bigotimes_{e \in E} \mathcal{H}_e.
\]

A \emph{site} $s$ is defined as a pair $(v, p) \in V \times P$ with the condition $v \subset p$. Associated with each \emph{site} $s$ and $g \in G$, we will define two types of local operators, $A_g(s)$ and $B_h(s)$. 

For a site $s = (v, p)$ and an element $g \in G$, the operator $A_g(s)$ acts on edges in $\operatorname{star}(v)$, where $\operatorname{star}(v)$ represents the set of edges to which the vertex $v$ is incident. The action of $A_g(s)$ on an edge in $\operatorname{star}(v)$ depends on the orientation of the edge with respect to $v$: it is the left multiplication by $g$ if the edge is oriented away from $v$, and the right multiplication by $g^{-1}$ if it points towards $v$.

Additionally, for each site $s = (v, p)$ and an element $h \in G$, the operator $B_h(s)$ is defined to act on edges that form the boundary $\partial p$ of the face $p$. The action of $B_h(s)$ is determined by a specific product of group elements. To compute this product, one starts at vertex $v$ and travels around the boundary of $p$ in the counterclockwise direction. The group elements associated with the edges are multiplied in the order they are encountered. However, if the orientation of an edge is opposite to the direction of travel, the inverse of the associated group element on that edge is used in the product. The operator $ B_h(s)$ acts as the identity if $h$ equals this product. If $h$ does not match the product, $ B_h(s)$ acts as \( 0 \). An example of the action of $B_h$ on a site $s=(v,p)$ is shown in Figure \ref{operatorf1}.
\begin{figure}
  \centering
  \includegraphics[scale=1]{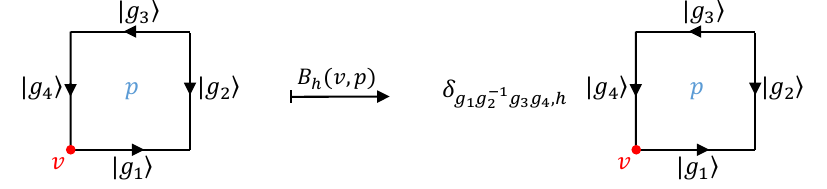}
\caption{Example of the action of operator $B_h$ on the site $s=(v,p)$.}
 \label{operatorf1}
\end{figure}

\begin{remark}
The operators $A_g(s)$ and $B_h(s)$ satisfy the relations:
\begin{multicols}{2} 
\begin{itemize}
    \item[$\imath)$] $A_{g_1}(s)A_{g_2}(s) = A_{g_1g_2}(s)$
    \item[$\imath \imath)$] $B_{h_1}(s)B_{h_2}(s) = \delta_{h_1,h_2}B_{h_1}(s)$
    \item[$\imath \imath \imath)$] $A_{g}(s)B_{h}(s) = B_{ghg^{-1}}A_g(s)$
    \item[$\imath v)$] $A_e(s)= \sum_{h\in G} B_h(s) = \operatorname{Id}$,
\end{itemize}
\end{multicols}for all $g_1,g_2, h_1,h_2\in G$.
\end{remark}

By definition, the action of $A_g(s)$ for a site $s=(v,p)$ does not depend on the face $p$, but only on the vertex $v$, that is $A_g(s)=A_g(v)$. The operators $B_h(s)$ depend on both $p$ and $v$, but that is not the case when $h=e$ because a product of group elements resulting in the identity is a cyclic property, that is, $B_e(s)=B_e(p)$. The vertex operator, denoted by $A(v)$, and the  face operator, denoted by $B(p)$, are defined as follows:
\begin{align*}
    A(v) := \frac{1}{|G|}\sum_{g \in G} A_g(v) && \text{and} && B(p) := B_e(p).
\end{align*}

\begin{remark}
For each $v_1, v_2 \in V$ and $p_1, p_2 \in P$, the operators satisfy:
\begin{itemize}
    \item[$\imath)$] Commutativity: $A(v_1)A(v_2) = A(v_2)A(v_1)$, $B(p_1)B(p_2) = B(p_2)B(p_1)$, and $A(v_1)B(p_1) = B(p_1)A(v_1)$.
    \item[$\imath \imath)$] Idempotence: $A(v_1)^2 = A(v_1)$, $B(p_1)^2 = B(p_1)$.
\end{itemize}    
\end{remark}

The quantum code corresponding to the lattice $\mathcal{L}$ and the group $G$ is defined by the subspace of the \emph{ground states}  $V_{\text{gs}} \subset \mathcal{H}_{\text{tot}}$, given by 
$$
V_{\text{gs}} = \left\{ \ket{\psi} \in \mathcal{H}_{\text{tot}} : A(v) \ket{\psi} = \ket{\psi}, \, B(p) \ket{\psi} = \ket{\psi}, \, \forall v \in V, \, p \in P \right\}.
$$
In \cite{anyons}, it is stated that the space $V_{\text{gs}}$ only depends on the topology of $\Sigma$ and that the dimension of the ground states is equal to the number of orbits in $\operatorname{Hom}(\pi_1(\Sigma),G)$ under
the conjugation $G$-action. Moreover, in \cite{shaw2020}, it is proved mathematically that $$V_{\text{gs}} \subset \mathcal{H}_{\text{tot}}$$ is a quantum error-correcting code for every finite group.

\section{Kitaev's model for abelian groups}\label{acase}
Let $G$ be an abelian group and $\widehat{G}$ the abelian group of linear characters on $G$, that is, $\widehat{G}$ consists of group morphisms from $G$ to $U(1)$. 
Consider the operators on $\mathcal{H}_e = \mathbb{C}[G] = \operatorname{span}\{ \ket{g} : g \in G \}$, defined by $\{X_g : g \in G\}$ where $X_g\ket{h} = \ket{g+h}$, and the set of operators $\{Z_\gamma : \gamma \in \widehat{G} \}$, such that $Z_\gamma\ket{g} = \gamma(g)\ket{g}$. Furthermore, as the definition of the Kitaev's model is dependent on the lattice orientation, we will also consider the following functions.

\begin{definition} \label{signo}
Let $(e,v) \in E \times V $, 
where $e \in star(v)$,  and $(e,p) \in P \times V $ where $e \in \partial p$. We define the sign functions
\begin{align}
\varepsilon(e,v) :=\begin{cases} ~ ~ 1, & \text{if } e \text{ is pointed away from } v,\\-1, & \text{if } e \text{ is pointed to } v. \end{cases}
\end{align}

\begin{align}
    \varepsilon(e,p) := \begin{cases}
        \phantom{-}1, & \text{if } e \text{ is oriented in the counterclockwise direction along } \partial p, \\
        -1, & \text{if } e \text{ is oriented in the clockwise direction along } \partial p.
    \end{cases}
\end{align}

\end{definition}

We can express the operator $A_g$ using the $X_g$ operator and the function $\varepsilon(e,v)$. $A_g(v)$ applies either the operator $X_g$ or the operator $X_{-g}$ to the edges adjacent to vertex $v$, depending on the orientation of the edge. It acts as the identity operator on the other edges of the lattice. Therefore, we have:
\begin{align} \label{vertexoperatorg}
    A_g(v) = \left( \bigotimes_{e \in star(v)}X_g^{\varepsilon(e,v)}\right) \otimes \left( \bigotimes_{e \in E\setminus star(v)}I_e \right)
\end{align}

See Figure \ref{groupv2} for an example of the action of $A_g(v)$.
\begin{figure}[h!]
\centering
	  \centering  \includegraphics[scale=1]{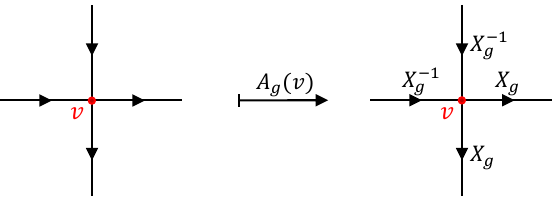}
\caption{The action of operator $A_g(v)$ on $v \in V$.}
\label{groupv2}
\end{figure}

The vertex operator $A(v)$ is given by:
\begin{align*}
A(v)  = \frac{1}{|G|}\sum_{g\in G} A_g(v)\
 = \frac{1}{|G|}\sum_{g\in G} \left( \left( \bigotimes_{e \in \text{star}(v)}X^{\varepsilon(e,v)}_g\right) \otimes \left( \bigotimes_{e \in E \setminus star(v)}I_e \right) \right)
\end{align*}

For every linear character $\gamma \in \widehat{G}$ and face $p$, we define a unitary operator $\widehat{B_\gamma}(p)$, where

\begin{align} \label{faceoperatorgamma}
    \widehat{B}_\gamma(p) := \left( \bigotimes_{e \in \partial p}Z^{\varepsilon(e,p)}_{\gamma}\right) \otimes \left( \bigotimes_{e \in E\setminus \partial p}I_e \right).
\end{align}
An example of the action of $\widehat{B}_\gamma$ on a face $p$ is shown in Figure \ref{operatorf2}.
\begin{figure}
  \centering
  \includegraphics[scale=1]{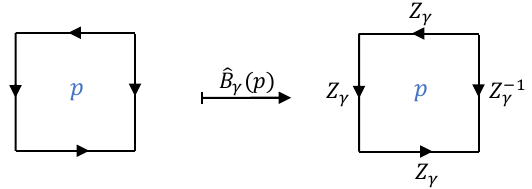}
\caption{The action of operator $\widehat{B}_\gamma$ on the face $p$.}
 \label{operatorf2}
\end{figure}

Moreover, because the group is abelian, the site operators $B_g(s)$  only depend on the face, and they can be expressed using the operators $\{\widehat{B}_\gamma:\gamma\in \widehat{G}\}$,
\begin{align} \label{faceoperatorh}
B_g(p) & =  \frac{1}{|G|}\sum_{\gamma \in \widehat{G}} \overline{\gamma(g)} \widehat{B}_\gamma(p).
\end{align}
The face operator $B(p)$ can be written in terms of the $Z_\gamma$ operator as follows
\begin{align*}
B(p) & = B_e(p)\\
& =   \frac{1}{|G|}\sum_{\gamma \in \widehat{G}} \widehat{B}_\gamma\\
& =  \frac{1}{|G|}\sum_{\gamma \in \widehat{G}} \left( \left( \bigotimes_{e \in \partial p}Z^{\varepsilon(e,p)}_\gamma\right) \otimes \left( \bigotimes_{e \in E \setminus \partial p}I_e \right) \right). 
\end{align*}

The following result allows us to express the ground states only in terms of the operators $A_g$ and $\widehat{B}_\gamma$, which  are expressed in terms of  $X_g$ and $Z_\gamma$ respectively.

\begin{proposition} \label{bprimas}
$\ket{\psi}$ belongs to $V_{gs} $ if and only if for every $g \in G$ and $\gamma \in \widehat{G}$,
every vertex $v \in V$ and face $p \in P$, we have that $A_g(v)\ket{\psi} =\ket{\psi}$ and $\widehat{B}_\gamma(p)\ket{\psi} =\ket{\psi}$. 
\end{proposition}
\begin{proof}
    Suppose $\ket{\psi} \in V_{gs} $, then for the operators $A_g(v)$ we have
    \begin{align*}
        A_g(v)\ket{\psi} & = A_g(v)A(v)\ket{\psi}  =  A(v)\ket{\psi} = \ket{\psi}
    \end{align*}
   Similarly for the operators $B_\gamma(p)$,
    \begin{align*}
        \widehat{B}_\gamma(p)\ket{\psi} & = \widehat{B}_\gamma (p)B(p)\ket{\psi} =  B(p)\ket{\psi} = \ket{\psi}
    \end{align*}
    For the other direction, we assume that $A_g(v)\ket{\psi} =\ket{\psi}$ and $\widehat{B}_\gamma(p)\ket{\psi} =\ket{\psi}$ for every $g \in G$ and $\gamma \in \widehat{G}$. Because $A(v) = \frac{1}{|G|}\sum_{g \in G} A_g(v)$ then $A(v) \ket{\psi} = \ket{\psi}$ and since $B(p) = \frac{1}{|G|}\sum_{\gamma \in \widehat{G}} \widehat{B}_\gamma(p)$ then $B(p)\ket{\psi} = \ket{\psi}$ and so $\ket{\psi} \in V_{gs} $.
\end{proof}

\section{Error detection and Syndrome extraction} \label{secciondeteccion}

As discussed in \cite[Section 3]{shaw2020}, the space of ground states, $V_{gs}$, is a quantum error-correcting code. To explain how to correct errors, we first need to detect errors and collect certain information, called the syndrome. This section will discuss the detection process and the extraction of the syndrome.

An error on a vector $\ket{\psi} \in V_{gs} \subset \cH_{tot}$ occurs when an operator acts on $\ket{\psi}$ and transforms it. It is known that if we can correct a set of errors then we are able to correct any error arising from a linear combination of these errors. Therefore, to correct errors produced by any operator, it suffices to correct the errors produced by a basis of the space of operators. In the discussions below, we will fix a basis for the space of operators on the total space, and assume without loss of generality that an error is always one of the basis elements.

\begin{proposition}\label{base}
An (orthogonal) basis for the space of operators on the total space $\mathcal{H}_{\text{tot}}$ is given by 
\[
\left\{ \bigotimes_{e \in E} X_{g_e} Z_{\gamma_e} : g_e \in G, \gamma_e \in \widehat{G} \right\},
\]
where $X_{g_e}$ and $Z_{\gamma_e}$ denote the local operators $X_{g}$ and $Z_{\gamma}$, respectively, acting on $\mathcal{H}_e$.
\end{proposition}
\begin{proof}
It is enough to prove that $\{ X_gZ_\gamma : g \in G, \gamma \in \widehat{G} \}$ forms an orthogonal basis of $\mathcal{H}_e$. Consider the Hilbert-Schmidt inner product, defined for any two operators $A$ and $B$ as
\begin{align}
\langle A,B \rangle_{\text{HS}} := \operatorname{Tr}(A^{\dagger}B) = \sum_{g \in G} \langle g|A^\dagger B|g \rangle.    
\end{align}
Let us prove that the set $\{ X_gZ_\gamma : g \in G, \gamma \in \widehat{G} \}$, is orthogonal with respect to the Hilbert-Schmidt inner product. Consider the following calculation 
\begin{align*}
    \langle X_gZ_\gamma, X_hZ_\beta \rangle_{\text{HS}} & = \sum_{f \in G} \langle f | (X_gZ_\gamma)^\dagger X_hZ_\beta | f \rangle \\
    & = \overline{\gamma(h-g)} \sum_{f \in G} \langle f | (X_{(-g)} X_h Z_{\overline{\gamma}}Z_\beta | f \rangle  \quad \quad \quad \quad \text{because } Z_\gamma X_h = \gamma(h) X_h Z_\gamma\\
    & = \overline{\gamma(h-g)} \sum_{f \in G} \overline{\gamma(f)}\beta(f) \langle f | -g+h+f \rangle \\
    & = \delta_{g,h}\overline{\gamma(h-g)}  \sum_{f \in G} \overline{\gamma(f)}\beta(f) \\
   & = \delta_{\gamma,\beta} \delta_{g,h}\overline{\gamma(h-g)}|G|,
\end{align*}
The last equality is derived from the orthogonality relations of irreducible characters. Consequently, the operators $X_gZ_\gamma$ form a basis for the space of $|G| \times |G|$ complex matrices.
\end{proof}


For error detection, consider the following projector operators. Given $(g, \gamma) \in G \times \widehat{G}$, and a site $s = (v, p) \in V \times P$, define
\begin{align}
    P_{(g, \gamma)}(s) := \frac{1}{|G|}B_g(p)\sum_{h \in G} \overline{\gamma(h)} A_h(v).
\end{align}where $A_h(v)$ and $B_g(p)$ are the operators in \ref{vertexoperatorg} and \ref{faceoperatorh} respectively.

\begin{proposition}
Let $s$ be a site. The family of operators $\{P_{(g,\gamma)}(s):(g,\gamma)\in G\times \widehat{G}\}$ forms a complete set of projectors of $\cH_{tot}$, which means
\begin{enumerate}
    \item $P_{(g,\gamma)}^2=P_{(g,\gamma)}$,
    \item $\sum_{(g,\gamma)\in G\times \widehat{G}} P_{(g,\gamma)} =\operatorname{Id}$
    \item $P_{(g,\gamma)}P_{(h,\beta)} = \delta_{g,h} \delta_{\gamma,\beta}P_{(g,\gamma)}$ 
\end{enumerate}
\end{proposition}
\begin{proof}
This is a straightforward computation and will be omitted for brevity.
\end{proof}

Let us see how projectors $P_{(g,\gamma)}$ help us to detect errors on the ground state $V_{gs}$. Let $s =(v,p)$ be a site  of the lattice and consider the operators $P_{(g,\gamma)}(s)$. For each site \( s \) and vector \( \ket{\psi} \) in \( V_{gs} \), we have
\[
    P_{(g,\gamma)}(s)\ket{\psi} =\delta_ {g,0}\delta_{\gamma,\mathbf{1}}\ket{\psi},
\]
where \( \mathbf{1} \in \widehat{G} \) is the trivial character. Moreover, let $\cP : \cH_{tot} \to V_{gs}$ be the projector onto $V_{gs}$, then
\begin{align}\label{proyector}
\cP = \bigotimes_{(v,p) \in V \times P} P_{0,\mathbf{1}}(v,p).
\end{align}
If a single error occurs on an edge \(e \in E\) of the lattice \(\mathcal{L}\), specifically when an operator of the form \(X_hZ_\beta\) acts on edge \(e\), it transforms the vector \(\ket{\psi} \in V_{gs}\) into \((X_hZ_\beta)_e \ket{\psi}\).
 We know that edge \( e \) connects two vertices and two faces; as seen in Figure \ref{unerror}. Moreover, \( e \) points to one of its two vertices; in this case, it points to \( v_2 \). On the other hand, \( e \) is oriented in the counterclockwise direction relative to some of the faces; in Figure \ref{unerror} it is the face \( p_1 \). 
\begin{figure}[h]
    \centering
  \includegraphics[scale=1]{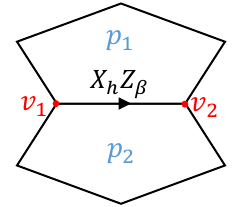}
   \caption{The error $X_hZ_\beta$ occurs on the edge that connects the vertices $v_1$ and $v_2$. The edge also belongs to the boundary of the faces $p_1$ and $p_2$}
   \label{unerror}
\end{figure}

Hence,  for the site $(v_2,p_1)$,
\begin{align*}
    P_{(g,\gamma)}(v_2,p_1)(X_hZ_\beta)_e\ket{\psi} & = \left( \frac{1}{|G|}\sum_{k \in G}\overline{\gamma(k)}B_g(p_1)A_k(v_2)\right) (X_hZ_\beta)_e\ket{\psi}\\
    & = (X_h)_e \left( \frac{1}{|G|}\sum_{k \in G}\overline{\gamma(k)}B_{g-h}(p_1)A_k(v_2)\right)(Z_\beta)_e\ket{\psi}\\
    & = (X_hZ_\beta)_e\left( \frac{1}{|G|}\sum_{k \in G}\overline{\gamma(k)}\overline{\beta(-k)}B_{g-h}(p_1)A_k(v_2)\right)\ket{\psi}\\
    & = (X_hZ_\beta)_e\left( \frac{1}{|G|}\sum_{k \in G}\overline{\gamma(k)}\beta(k)A_k(v_2)\right)B_{g-h}(p_1)\ket{\psi}\\
    & = \delta_{g,h}(X_hZ_\beta)_e\left( \frac{1}{|G|}\sum_{k \in G}\overline{\gamma(k)}\beta(k)A_k(v_2)\right)\ket{\psi}\\
    & =\delta_{g,h}\delta_{\gamma,\beta}(X_hZ_\beta)_e\ket{\psi}.
\end{align*}
The above calculation applies to the case depicted in Figure \ref{unerror}. In the same way, we have have similar results for all other possible scenarios, that is, different orientations of the edge with respect to the site. Likewise, this calculation demonstrates the behavior of the operators $P_{g,\gamma}$ when there is an error on a single edge. Proposition \ref{long} extends these calculations and illustrates the behavior of the $P_{g,\gamma}$ operators in situations where errors affect multiple edges.
\begin{proposition}\label{long}
Let $s=(v,p)$ be a site of $\cL$, $\cP : \cH_{tot} \to V_{gs}$ be the projector onto $V_{gs}$ and $R = \bigotimes_{e \in E} X_{g_e}Z_{\gamma_e}$ be an error operator. Then $P_{g,\gamma}(v,p) R \cP = \delta_{g,g'}\delta_{\gamma,\gamma '}R \cP$ where
\begin{align}
    \gamma ' = \prod_{e \in star(v)}\overline{\gamma_e}^{\varepsilon(e,v)} && g' = \sum_{e \in \partial p}\varepsilon(e,p)(g_e)
\end{align}
\end{proposition}
\begin{proof} Let $e$ an edge with error $X_{g_e}Z_{\gamma_e}$. If $e \not\in \partial p$ then $B_g(p) (X_{g_e}Z_{\gamma_e}) = (X_{g_e}Z_{\gamma_e}) B_g(p) $, and if $e \in \partial p$, then $B_g(p) (X_{g_e}Z_{\gamma_e}) = (X_{g_e}Z_{\gamma_e}) B_{g + \varepsilon(e,p)g_e}(p) $. Furthermore, if $e \not\in star(v)$, then $$\left( \frac{1}{|G|}\sum_{h \in G}\overline{\gamma(h)}A_h(v)\right) (X_{g_e}Z_{\gamma_e}) = (X_{g_e}Z_{\gamma_e})\left( \frac{1}{|G|}\sum_{h \in G}\overline{\gamma(h)}A_h(v)\right).$$
However, if $e \in star(v)$,
$$\left( \frac{1}{|G|}\sum_{h \in G}\overline{\gamma(k)}A_h(v)\right) (X_{g_e}Z_{\gamma_e}) = (X_{g_e}Z_{\gamma_e})\left( \frac{1}{|G|}\sum_{h \in G}\overline{\gamma(h)}~\overline{\gamma_e(h)}^{\varepsilon(e,v)}A_h(v)\right).$$
So \begin{align*}
    P_{(g,\gamma)}(v,p)R & = P_{(g,\gamma)} (v,p)\bigotimes_{e \in E}(X_{g_e}Z_{\gamma_e})\\
    & = \left( \frac{1}{|G|}\sum_{h \in G}\overline{\gamma(h)}A_h(v)B_{g}(p)\right)\bigotimes_{e \in E}(X_{g_e} Z_{\gamma_e}) \\
   & = \bigotimes_{e \in E}(X_{g_e}Z_{\gamma_e}) \left( \frac{1}{|G|}\sum_{h \in G}\overline{\gamma(h)}\gamma'(h)A_h(v)B_{g- g'}(p)\right) ,
\end{align*}
where $\gamma ' = \prod_{e \in star(v)}\overline{\gamma_e}^{\varepsilon(e,v)}$ and $ g' = \sum_{e \in \partial p}\varepsilon(e,p)(g_e)$. Also, since $\cP$ is the projector onto $V_{gs}$, for every face $p \in P$, we have $B_{h}(p)\cP = \delta_{h,e} \cP$, and for every vertex $v \in V$, $A_h(v) \cP = \cP$. Therefore,
\begin{align*}
    P_{(g,\gamma)}(v,p) R \cP & =  \bigotimes_{e \in E}(X_{g_e}Z_{\gamma_e}) \left( \frac{1}{|G|}\sum_{h \in G}\overline{\gamma(h)}\gamma'(h)A_h(v)B_{g- g'}(p)\right)  \cP\\
    & =  \bigotimes_{e \in E}(X_{g_e}Z_{\gamma_e}) \left( \frac{1}{|G|}\sum_{h \in G}\overline{\gamma(h)}\gamma'(h)A_h(v)\right)B_{g- g'}(p)\cP\\
    & = \delta_{g,g'} \bigotimes_{e \in E}(X_{g_e}Z_{\gamma_e}) \left( \frac{1}{|G|}\sum_{h \in G}\overline{\gamma(h)}\gamma'(h)A_h(v)\right)  \cP\\
    & =\delta_{g,g'}\delta_{\gamma,\gamma'} \bigotimes_{e \in E}(X_{g_e}Z_{\gamma_e})\\
    & =  \delta_{g,g'}\delta_{\gamma,\gamma'} 
 R \cP.
\end{align*}
\end{proof}
 Proposition \ref{long} implies that for a given error operator $R$ acting on $V_{gs}$, and for a specific site $s = (v, p)$, there exists a unique projector operator $P_{(g,\gamma)}(v,p)$ such that $P_{(g,\gamma)}(v,p)R \cP = R\cP$. This particular operator encodes error information in its indices: the first index $g$ captures information about the $X$-type errors on edges in the boundary of the face $p$, while the second index $\gamma$ captures information about $Z$-type errors on edges that are incident to $v$. Let us see this in an example.

\begin{example} \label{exdeteccion}
Let $G=\Z_2 \times \Z_4$, then $\widehat{G}= \{ \gamma_{a,b}: a \in \Z_2, b \in \Z_4\}$ where \begin{align*}
    \gamma_{a,b}(1,0)= (-1)^a, && \gamma_{a,b}(0,1)= i^b
\end{align*}
Consider the vertex, face, and edges labeled on the left side of Figure \ref{deteccion}  and suppose that on each of the labeled edges, an error $R$ applies the operators shown on the right of Figure \ref{deteccion}. The operator $X_{(1,1)}Z_{\gamma_{1,0}}$ acts on the edge $e_1$ and since $e_1$ is pointed to $v$ then $\overline{\gamma_1}^{\varepsilon(e_1,v)} = \gamma_{1,0}$, on the other hand, $e_2$ is pointed away from $v$ so $\overline{\gamma_2}^{\varepsilon(e_2,v)} = (\gamma_{0,3})^{-1}$, continuing in this way we have $$ \gamma ' = \prod_{e \in star(v)}\overline{\gamma_e}^{\varepsilon(e,v)}  = \gamma_{1,0}(\gamma_{0,3})^{-1}(\gamma_{1,3})^{-1}\gamma_{0,0}= \gamma_{0,2}.$$
Similarly, but now looking at the $X$-type errors, we have $$g' = \sum_{e \in \partial p}\varepsilon(e,p)(g_e) = (0,0)-(1,0)+(0,2)+(0,0) = (1,2) .$$
Then according to Proposition \ref{long} the only operator $P_{g,\gamma}$ on the site $(v,p)$ that acts as the identity on $R\cP$ is the operator $P_ {(1,2),\gamma_{0,2}}(v,p)$.
\end{example}

\begin{figure}[h]
    \centering
  \includegraphics[scale=1]{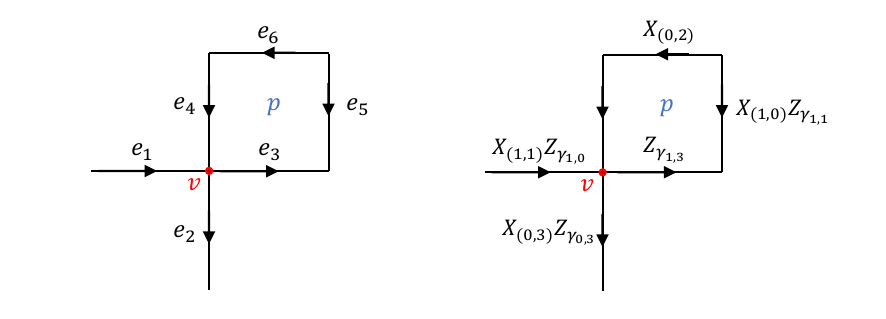}
   \caption{On the left we indicate the vertex $v$, the face $p$ and the edges $e_1,e_2,e_3,e_4,e_5,e_6$. On the right, we see the errors that have acted on each of the edges.}
   \label{deteccion}
\end{figure}

Note that for sites $s = (v,p)$ and $s' = (v,p')$ sharing the same vertex $v$, the following holds: Given any error operator $R$, if $P_{g,\gamma}(v,p)(R\cP)=P_{g',\gamma'}(v,p')(R\cP) = R\cP$, then it must be that $\gamma = \gamma'$. Similarly, consider sites $s = (v,p)$ and $s' = (v',p)$ with the same face $p$. If $P_{g,\gamma}(v,p)(R\cP) =P_{g',\gamma'}(v',p)(R\cP) = R\cP$, then it follows that $g = g'$. This implies that for the operator $P_{g,\gamma}(v,p)$, which acts as the identity on $R\cP$, the index $\gamma$ depends only on the vertex $v$, and the index $g$ depends only on the face $p$.

\begin{definition} \label{defsymdrome}
Let $\ket{\psi} \in V_{gs}$ and $R$ be an error operator such that $R\ket{\psi} = \ket{\psi'}$.
\begin{itemize}
    \item Let $v$ be a vertex, and $P_{g,\gamma}$ such that $P_{g,\gamma}(v,p)\ket{\psi'} = \ket{\psi'}$, then we define the syndrome of the error $R$ on the vertex $v$ as $Syn_{R}(v) =\gamma$.
    \item Let $p$ be a face, and $P_{g,\gamma}$ such that $P_{g,\gamma}(v,p)\ket{\psi'} = \ket{\psi'}$, then we define the syndrome of the error $R$ on the face $p$ as $Syn_{R}(p) =g$.
\end{itemize}
\end{definition}

\begin{remark}
    Note that although Definition \ref{defsymdrome} used a state $\ket{\psi} \in V_{gs}$ to determine the syndrome of the error $R$, this syndrome does not depend on $\ket{\psi}$.
\end{remark}

In conclusion, the error detection process for $R\ket{\psi} =\ket{\psi'}$ would be as follows.
\begin{itemize}
    \item[1)] For each site $(v,p)$ we measure the state $\ket{\psi'}$ with respect to the set of the projectors $\{P_{g,\gamma}(v,p)\ | \ g \in G, \gamma \in \hat{G}\}$. The outcome of the measurement tells us for which $(g,\gamma)$ we have $P_{g,\gamma}(v,p)\ket{\psi'} = \ket{\psi'}$.  
    \item[2)] Select the sites $(v,p)$ for which $P_{g,\gamma}(v,p)\ket{\psi'} = \ket{\psi'}$ with $P_{g,\gamma} \neq P_{0,\textbf{1}}$. (If there are sites where this occurs, it implies that an error has affected the total space). 
    \item[3)] For sites $(v,p)$ of the previous item, we label the vertex $v$ with $Syn_{R}(v)^{-1}$ and the face $p$ with the number $-Syn_{R}(p)$, (Taking the inverse will make sense later.).
\end{itemize}

\begin{figure}[h]
    \centering
  \includegraphics[scale=0.9]{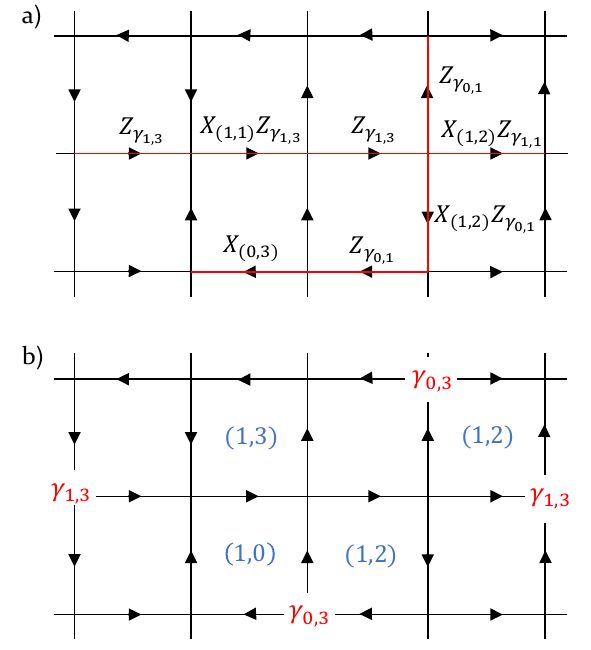}
   \caption{In a) are the errors that have occurred. In b) is the information extracted by the operators $P_{g,\gamma}$.}
   \label{deteccion2}
\end{figure} 

\begin{example}
Let $\cL$ a square lattice, and $G= \Z_2 \times \Z_4$. Figure \ref{deteccion2} in a) shows the errors that have occurred. The operators $P_{g,\gamma}$ give us the syndrome measurement, so in b), we see how we have labeled vertices $v$ with $Syn_{R}(v)^{-1}$, and faces $p$ with $-Syn_{R}(p)$. These labels give us information that we will use to correct the errors.
\end{example}

\section{Analysis of the error through its syndrome} \label{syndrome}

This section examines the information provided by the syndrome at each vertex and face, that is $Syn_{R}(v)$ and $Syn_{R}(p)$, respectively. An error $R =  \bigotimes_{e \in E}X_{ g_e}Z_{\gamma_e}$ can be seen as $R= R_XR_Z$ where $R_X = \bigotimes_{e \in E}X_{ g_e}$ and $R_Z =\bigotimes_{e \in E}Z_{\gamma_e}$. When $Syn_{R}(v) \neq \mathbf{1}$ for some vertex, then $R_Z \neq \operatorname{Id}$. Similarly, if $Syn_{R}(p) \neq 0$ for some face, then $R_X \neq \operatorname{Id}$. However, the converse statement does not generally hold. Thus, there exist errors where $R \neq \operatorname{Id}$ while $Syn_{R}(v) = \mathbf{1}$ and $Syn_{R}(p) = 0$ for all $p$ and $v$. To understand the limitations of the detection process, we will focus on characterizing the error operators that result in $Syn_{R}(v) = \mathbf{1}$ and $Syn_{R}(p) = 0$ for all $p$ and $v$.

First, for an error $R$, we must understand the subgraph of $\mathcal{L}$ composed of edges affected by $R$. Therefore, we will make some observations from graph theory.

\subsection{Preliminaries on graph theory}

First, we will recall the notion of a dual lattice and some basic definition and results from graph theory that will be important for our discussion.
 
\begin{definition}
The dual $\cL^*$ of a lattice $\cL$ on a surface $\Sigma$ is a lattice on $\Sigma$ whose vertices correspond to the faces of $\cL$ and whose edges correspond to the edges of $\cL$. The edge $e^* \in E(\cL^*)$ that corresponds to $e \in E(\cL)$ connects the vertices corresponding to the faces of $\cL$ whose boundary contains $e$.

A canonical embedding of the dual lattice \(\cL^*\) positions the dual vertex \(p^*\) inside each face \(p\) of \(\cL\). For any edge \(e\), which lies on the boundary between two faces \(p_1\) and \(p_2\), a curve is drawn from \(p_1^*\) to \(p_2^*\), ensuring these curves do not intersect. Each curve creates a dual edge \(e^*\). See Figure \ref{dual} for a visual representation of \(\cL^*\), where solid lines depict the lattice \(\cL\) and dashed lines represent \(\cL^*\).

\end{definition}
\begin{figure}[h!]
    \centering
\begin{tikzpicture}[x=0.65pt,y=0.65pt,yscale=-1,xscale=1]

\draw   (46,222) -- (62.5,167) -- (113.5,167) -- (130,222) -- cycle ;
\draw    (130,222) -- (256,222) ;
\draw    (113.5,167) -- (186,222) ;
\draw    (113.5,167) -- (197,167) ;
\draw    (197,167) -- (256,222) ;
\draw    (267,149) -- (256,222) ;
\draw    (196,118) -- (197,167) ;
\draw    (251,104) -- (196,118) ;
\draw    (267,149) -- (251,104) ;
\draw    (113.5,167) -- (141,118) ;
\draw    (141,118) -- (196,118) ;
\draw    (62.5,167) -- (50,123) ;
\draw    (50,123) -- (96,98) ;
\draw    (96,98) -- (141,118) ;
\draw    (96,98) -- (145,80) ;
\draw    (145,80) -- (196,118) ;
\draw    (145,80) -- (232,77) ;
\draw    (251,104) -- (232,77) ;
\draw    (96,98) -- (89,43) ;
\draw    (89,43) -- (153,43) ;
\draw    (153,43) -- (217,44) ;
\draw    (153,43) -- (145,80) ;
\draw    (217,44) -- (232,77) ;
\draw    (89,43) -- (50,123) ;
\draw   (90,135.5) .. controls (90,134.12) and (91.12,133) .. (92.5,133) .. controls (93.88,133) and (95,134.12) .. (95,135.5) .. controls (95,136.88) and (93.88,138) .. (92.5,138) .. controls (91.12,138) and (90,136.88) .. (90,135.5) -- cycle ;
\draw   (182,60.5) .. controls (182,59.12) and (183.12,58) .. (184.5,58) .. controls (185.88,58) and (187,59.12) .. (187,60.5) .. controls (187,61.88) and (185.88,63) .. (184.5,63) .. controls (183.12,63) and (182,61.88) .. (182,60.5) -- cycle ;
\draw   (78,88.5) .. controls (78,87.12) and (79.12,86) .. (80.5,86) .. controls (81.88,86) and (83,87.12) .. (83,88.5) .. controls (83,89.88) and (81.88,91) .. (80.5,91) .. controls (79.12,91) and (78,89.88) .. (78,88.5) -- cycle ;
\draw   (204,95.5) .. controls (204,94.12) and (205.12,93) .. (206.5,93) .. controls (207.88,93) and (209,94.12) .. (209,95.5) .. controls (209,96.88) and (207.88,98) .. (206.5,98) .. controls (205.12,98) and (204,96.88) .. (204,95.5) -- cycle ;
\draw   (85.5,197) .. controls (85.5,195.62) and (86.62,194.5) .. (88,194.5) .. controls (89.38,194.5) and (90.5,195.62) .. (90.5,197) .. controls (90.5,198.38) and (89.38,199.5) .. (88,199.5) .. controls (86.62,199.5) and (85.5,198.38) .. (85.5,197) -- cycle ;
\draw   (140,205.5) .. controls (140,204.12) and (141.12,203) .. (142.5,203) .. controls (143.88,203) and (145,204.12) .. (145,205.5) .. controls (145,206.88) and (143.88,208) .. (142.5,208) .. controls (141.12,208) and (140,206.88) .. (140,205.5) -- cycle ;
\draw   (185,197.5) .. controls (185,196.12) and (186.12,195) .. (187.5,195) .. controls (188.88,195) and (190,196.12) .. (190,197.5) .. controls (190,198.88) and (188.88,200) .. (187.5,200) .. controls (186.12,200) and (185,198.88) .. (185,197.5) -- cycle ;
\draw   (160,142.5) .. controls (160,141.12) and (161.12,140) .. (162.5,140) .. controls (163.88,140) and (165,141.12) .. (165,142.5) .. controls (165,143.88) and (163.88,145) .. (162.5,145) .. controls (161.12,145) and (160,143.88) .. (160,142.5) -- cycle ;
\draw   (233,156.5) .. controls (233,155.12) and (234.12,154) .. (235.5,154) .. controls (236.88,154) and (238,155.12) .. (238,156.5) .. controls (238,157.88) and (236.88,159) .. (235.5,159) .. controls (234.12,159) and (233,157.88) .. (233,156.5) -- cycle ;
\draw   (143,104.5) .. controls (143,103.12) and (144.12,102) .. (145.5,102) .. controls (146.88,102) and (148,103.12) .. (148,104.5) .. controls (148,105.88) and (146.88,107) .. (145.5,107) .. controls (144.12,107) and (143,105.88) .. (143,104.5) -- cycle ;
\draw   (119,65.5) .. controls (119,64.12) and (120.12,63) .. (121.5,63) .. controls (122.88,63) and (124,64.12) .. (124,65.5) .. controls (124,66.88) and (122.88,68) .. (121.5,68) .. controls (120.12,68) and (119,66.88) .. (119,65.5) -- cycle ;
\draw  [dash pattern={on 0.84pt off 2.51pt}]  (145.5,104.5) -- (162.5,140) ;
\draw  [dash pattern={on 0.84pt off 2.51pt}]  (187.5,197.5) -- (215,223) ;
\draw  [dash pattern={on 0.84pt off 2.51pt}]  (259,126.5) -- (233,156.5) ;
\draw  [dash pattern={on 0.84pt off 2.51pt}]  (209,95.5) -- (241.5,90.5) ;
\draw  [dash pattern={on 0.84pt off 2.51pt}]  (235.5,159) -- (261.5,185.5) ;
\draw  [dash pattern={on 0.84pt off 2.51pt}]  (95,135.5) -- (160,142.5) ;
\draw  [dash pattern={on 0.84pt off 2.51pt}]  (142.5,205.5) -- (187.5,200) ;
\draw  [dash pattern={on 0.84pt off 2.51pt}]  (88,199.5) -- (140,205.5) ;
\draw  [dash pattern={on 0.84pt off 2.51pt}]  (162.5,145) -- (187.5,195) ;
\draw  [dash pattern={on 0.84pt off 2.51pt}]  (204,95.5) -- (148,104.5) ;
\draw  [dash pattern={on 0.84pt off 2.51pt}]  (124,65.5) -- (184.5,58) ;
\draw  [dash pattern={on 0.84pt off 2.51pt}]  (83,88.5) -- (121.5,65.5) ;
\draw  [dash pattern={on 0.84pt off 2.51pt}]  (80.5,91) -- (92.5,133) ;
\draw  [dash pattern={on 0.84pt off 2.51pt}]  (92.5,138) -- (88,199.5) ;
\draw  [dash pattern={on 0.84pt off 2.51pt}]  (95,135.5) -- (143,104.5) ;
\draw  [dash pattern={on 0.84pt off 2.51pt}]  (124,65.5) -- (145.5,102) ;
\draw  [dash pattern={on 0.84pt off 2.51pt}]  (184.5,63) -- (206.5,93) ;
\draw  [dash pattern={on 0.84pt off 2.51pt}]  (206.5,95.5) -- (235.5,154) ;
\draw  [dash pattern={on 0.84pt off 2.51pt}]  (190,197.5) -- (235.5,154) ;
\draw  [dash pattern={on 0.84pt off 2.51pt}]  (224.5,60.5) -- (184.5,58) ;
\draw  [dash pattern={on 0.84pt off 2.51pt}]  (185,43.5) -- (184.5,63) ;
\draw  [dash pattern={on 0.84pt off 2.51pt}]  (121,43) -- (121.5,63) ;
\draw  [dash pattern={on 0.84pt off 2.51pt}]  (69.5,83) -- (78,88.5) ;
\draw  [dash pattern={on 0.84pt off 2.51pt}]  (56.25,145) -- (90,135.5) ;
\draw  [dash pattern={on 0.84pt off 2.51pt}]  (53.67,197.67) -- (90.5,197) ;
\draw  [dash pattern={on 0.84pt off 2.51pt}]  (88,194.5) -- (89.67,221.67) ;
\draw  [dash pattern={on 0.84pt off 2.51pt}]  (142.5,205.5) -- (156.33,221.67) ;
\draw  [dash pattern={on 0.84pt off 2.51pt}]  (162.5,145) -- (233,156.5) ;
\end{tikzpicture}
   \caption{A generic lattice $\cL$ (solid lines) and its dual lattice $\cL^*$ (dashed lines).}
   \label{dual}
\end{figure}
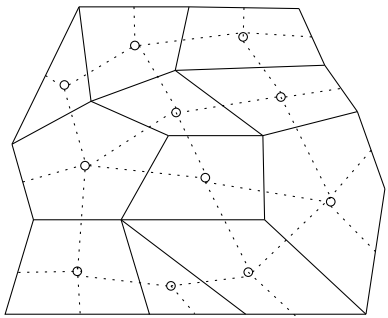

An orientation on the lattice \(\cL\) naturally induces a orientation of its dual \(\cL^*\). For an edge \(e\) in \(\cL\), its dual \(e^*\) in \(\cL^*\) is oriented so that rotating \(e\) counterclockwise to overlap with \(e^*\) aligns their orientations. Figure \ref{fig:orientation dual lattice} shows the orientations for \(e\) in \(\cL\) (solid lines) and \(e^*\) in \(\cL^*\) (dashed lines).

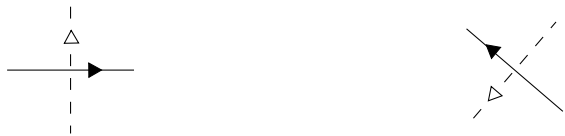
\begin{figure}[h!]
    \centering
    \begin{tikzpicture}[x=0.60pt,y=0.60pt,yscale=-1,xscale=1]

\draw  [dash pattern={on 4.5pt off 4.5pt}]  (64,20) -- (64,100) ;
\draw    (24,60) -- (104,60) ;
\draw    (313.63,33.96) -- (374.37,86.04) ;
\draw  [dash pattern={on 4.5pt off 4.5pt}]  (317.96,90.37) -- (370.04,29.63) ;
\draw  [fill={rgb, 255:red, 0; green, 0; blue, 0 }  ,fill opacity=1 ] (83.5,59.96) -- (75.54,64.54) -- (75.46,55.54) -- cycle ;
\draw  [fill={rgb, 255:red, 0; green, 0; blue, 0 }  ,fill opacity=1 ] (326.05,44.06) -- (335.09,45.66) -- (329.38,52.62) -- cycle ;
\draw  [fill={rgb, 255:red, 255; green, 255; blue, 255 }  ,fill opacity=1 ] (327.4,79.36) -- (329.19,70.36) -- (336.03,76.21) -- cycle ;
\draw  [fill={rgb, 255:red, 255; green, 255; blue, 255 }  ,fill opacity=1 ] (64.42,35) -- (69.08,42.9) -- (60.08,43.09) -- cycle ;
\end{tikzpicture}
    \caption{Orientation induced on the dual lattice.}
    \label{fig:orientation dual lattice}
\end{figure}

\begin{definition} Let $\cL$ be a lattice on a surface $\Sigma$. 

\begin{itemize}
\item  A \emph{path} is a sequence of edges \((e_1, e_2, \ldots, e_{n-1})\) for which there exists a sequence of vertices \((v_1, v_2, \ldots, v_n)\) such that \(\partial(e_i) = \{v_i, v_{i + 1}\}\) for \(i = 1, 2, \ldots, n-1\), where $v_i \neq v_j$ if $i \neq j$.
\item A \emph{cycle} $\cC$ is a sequence of edges \((e_1, e_2, \ldots, e_{n-1})\) for which there exists a sequence of vertices \((v_1, v_2, \ldots, v_n)\) such that \(\partial(e_i) = \{v_i, v_{i + 1}\}\) for \(i = 1, 2, \ldots, n-1\). For any $i \neq j$ , it holds that $v_i = v_j$ if and only if $i,j \in \{ 1,n\}$. 

\item A cycle is called \emph{contractible} if it bounds a disk in $\Sigma$. A subgraph $\G\subset \cL $ is \emph{contractible} if every cycle  $\cC \subseteq \G$ is  contractible.

\item  A \emph{bridge} is an edge of a graph whose deletion increases the number of connected components of the graph.

\item A \emph{bridgeless graph} is a graph that does not have any bridges.
\end{itemize}
\end{definition}

 \begin{notacion}
   Let $\G$ be a subgraph of $\cL$. Since in this work edges are denoted by the letter $e$, to avoid cumbersome notation, if an edge $e$ belongs to $\G$ we will write $e \in \G$ instead of $e \in E(\G)$. Similarly, as vertices are denoted by the letter $v$, if a vertex $v$ belongs to $\G$ we will write $v \in \G$ instead of $v \in V(\G)$.
\end{notacion}

Note that if $L$ is a bridgeless graph, then every edge in $L$ must be part of a cycle. Consequently, $L$ can be covered by a finite set of cycles $\{ \cC_i\}_{i=1}^{n}$, meaning that $L = \bigcup_{i=1}^{n} \cC_i$. In this situation, we refer to the set of cycles $\{ \cC_i\}_{i=1}^{n}$ as a \emph{cycle cover} of $L$. It is important to note that cycle covers are not unique. Moreover, to study different properties of graphs we are going to use a specific kind of covers.

\begin{definition}[Ear decomposition]
    An ear decomposition of a graph $\G$ is
    a sequence of subgraphs of $\G$, say $\G_1 \subset \cdots \subset \G_n$, such that $\G_1$ is a cycle, $\G_n = \G$, and $\G_{i+1} = \G_i \cup \rho_{i+1}$ where $\rho_{i+1}$ can be one of the two:
    \begin{itemize}
        \item[1)] A path internally disjoint from $\G_i$ with both ends in $\G_i$, or
        \item[2)] A cycle with only one vertex in $\G_i$.
    \end{itemize}    
\end{definition}

Proposition \ref{ear} is a well-known result in graph theory, and its proofs can be found in, for example, \cite[Theorem 5.2.4]{graphbook}.

\begin{proposition} \label{ear}
A graph $\G$ is bridgeless if and only if it has an ear decomposition. \qed
\end{proposition}

For a bridgeless graph $L$, we will construct a cycle cover using an ear decomposition. Let $\mathcal{G}_1 \subset \cdots \subset \mathcal{G}_n = L$ be an ear decomposition of $L$. Firstly, it is important to note that each of the subgraphs $\mathcal{G}_i$ is bridgeless, as the sequence $\G_1 \subset \cdots \subset \G_i$ constitutes an ear decomposition of $\G_i$. By applying Proposition \ref{ear}, we deduce that $\G_i$ is bridgeless.

Now, let us proceed to build a cycle cover for $L$. We initiate by taking the first cycle $\cC_1 := \G_1$. The remaining $\cC_i$'s will be defined recursively. Recall that $\G_{i} = \G_{i-1} \cup \rho_{i}$, where $\rho_i$ is either a cycle or a path. If $\rho_{i}$ is a cycle, we set $\cC_i := \rho_{i} $. If $\rho_{i}$ is a path with endpoints $v_1, v_2$ that belong to $\G_{i-1}$, and since $\G_{i-1}$ is bridgeless, there must be a path $\rho'_i \subset \G_{i-1}$ connecting $v_1$ with $v_2$. In such a case, we define  $\cC_i:= \rho_{i} \cup \rho'_{i}$. A covering constructed this way will be call an \emph{ear-cycle cover}.
 
\begin{proposition} \label{noloopforma}
Let $\G$ be a graph, such that every vertex has degree more than one. Then, $\G$ is either bridgeless, or there exists a maximal bridgeless subgraph $L \subset \G$ in which only one vertex of $L$ is incident to a bridge in $\G$.
\end{proposition}
\begin{proof}
   It is sufficient to prove the case when $\mathcal{G}$ has a single connected component. Suppose $\mathcal{G}$ is not bridgeless. Consider the subgraph of $\mathcal{G}$ consisting of the vertices and edges that belong to a cycle. Let $\mathcal{G}_1, \dots, \mathcal{G}_n$ be the connected components of this subgraph. By construction, each of these connected components is a maximal bridgeless subgraph, meaning that there exists no other bridgeless subgraph that properly contains any $\mathcal{G}_i$. Since $\mathcal{G}$ is connected, the $\mathcal{G}_i$'s are connected through paths composed of bridges, and, by hypothesis, every vertex has a degree greater than one. Consequently, the finals vertices of each of these paths belongs to some $\mathcal{G}_i$.

Consider the graph $\mathcal{T}$ with a vertex set $\{\mathcal{G}_1, \dots, \mathcal{G}_n \}$. There is an edge between $\mathcal{G}_i$ and $\mathcal{G}_j$ if and only if there exists a path composed of bridges in $\mathcal{G}$ that connects $\mathcal{G}_i$ with $\mathcal{G}_j$. By construction, every edge in $\mathcal{T}$ is a bridge; hence, $\mathcal{T}$ has no cycles. Therefore, $\mathcal{T}$ has at most $n-1$ edges, and since $\T$ has $n$ vertices, this implies that at least one vertex is incident to only one edge. Consequently, within $\mathcal{G}$, one of the subgraphs, $\mathcal{G}_i$, is connected only to another $\mathcal{G}_j$ through a path composed of bridges, meaning that only one of the vertices of $\mathcal{G}_i$ is incident to a bridge in $\mathcal{G}$.
 \end{proof}

\subsection{Behavior of an undetectable error}

Now we will present the main results of Section \ref{syndrome}, that correspond to Theorem \ref{multiloop} and Theorem \ref{multiloop contractil}. Theorem \ref{multiloop} describes the graph formed by edges affected by errors when the error is undetectable, i.e., when $Syn_{R}(v) = \mathbf{1}$ for all $v \in V$, and $Syn_{R}(p) = 0$ for all $p \in P$. 

\begin{definition}\hfill
    \begin{itemize}
        \item Given an error  $R= \bigotimes_{e \in E} X_{g_e}Z_{\gamma_e}$, its associated graph $\G(R)$, is defined as the subgraph of $\cL$ formed by the edges $e \in E$ such that $X_{g_e}Z_{\gamma_e} \neq \Id$.
            \item For an error  $R$,  we denote by $\G^*(R) \subset \cL^* $  the dual graph of the graph $\G(R)$, 
        \item For a subgraph $\G \subset \cL$, and an error $ R= \bigotimes_{e \in E} X_{g_e}Z_{\gamma_e}$ we define the restriction of $R$ on $\G$ as $$R|_{\G} := \left( \bigotimes_{e \in \G}X_{g_e}Z_{\gamma_e}\right) \otimes \left( \bigotimes_{e \not\in \G}I_e \right).$$
    \end{itemize}
\end{definition}

\begin{theorem}\label{multiloop} Let $R = R_XR_Z$ be an error on $V_{gs}$
\begin{itemize}
    \item[$\imath$)] If for all $v \in V$, $Syn_{R}(v) = \mathbf{1}$ then $\G(R_Z) \subset \cL$ is bridgeless.
    \item[$\imath \imath$)] If for all $p \in P$, $Syn_{R}(p) = 0$ then $\G^*(R_X) \subset \cL^*$ is bridgeless.
\end{itemize}    
\end{theorem} 

Theorem \ref{multiloop} is important for understanding the limitation of the detection. According to Theorem \ref{multiloop}, if an error $R = R_XR_Z$ occurs and its syndrome is trivial, meaning that $R$ is undetectable, then $\G(R_Z)$ and $\G^*(R_X)$ are bridgeless graph. This presents a problem for the code since the undetectability of such errors implies that they cannot be corrected. However, as will see with Theorem \ref{multiloop contractil}, this issue arises only if the bridgeless graph is non-contractible. For contractible bridgeless graphs, the errors will not affect $V_{gs}$; that is, these errors will not damage the code.

\begin{theorem}\label{multiloop contractil} Let $R = R_XR_Z$ be an error operator.
\begin{itemize}
    \item[$\imath$)] If $Syn_{R}(v) = \mathbf{1}$ for all $v \in V$, and $\G(R_Z)$ is a contractible graph. Then for all $\ket{\psi} \in V_{gs}$, $R_Z\ket{\psi} = \ket{\psi}$.
    \item[$\imath \imath$)] If $Syn_{R}(p) = 0$ for all $p \in P$, and $\G^*(R_X)$ is a contractible graph in the dual lattice. Then for all $\ket{\psi} \in V_{gs}$, $R_X\ket{\psi} = \ket{\psi}$.
\end{itemize}   
\end{theorem}

To prove Theorems \ref{multiloop} and \ref{multiloop contractil}, we require some results and observations. Therefore, the remainder of this section will be dedicated to their proofs.

\subsection{Proof of Theorems \ref{multiloop} and \ref{multiloop contractil}}

We have already defined syndromes for vertices and faces, providing information about $Z_\gamma$ and $X_g$ errors, respectively. Therefore, most of the forthcoming results will have to versions: one addressing $Z$-type errors involving the syndrome on vertices and graphs in the lattice, and other concerning $X$-type errors related to the syndrome on faces and graphs in the dual lattice. Consequently, we will proved the results only for $Z$-type errors. The proofs for $X$-type errors are analogous and will be omitted.

\begin{proposition}
For every error operator $R$, the syndrome of $R$ satisfies that
\begin{align}
    \prod_{v \in V}Syn_{R}(v)= \mathbf{1}, && \sum_{p \in P}Syn_{R}(p)= 0.
\end{align}
\end{proposition}
\begin{proof}
We will provide only the proof for the first equation, as the proof for the second follows a similar logic. Suppose that on an edge $e$, we have the error $Z_{\gamma}$. This edge is incident to only two vertices, denoted $v_1$ and $v_2$. Consequently, the error $Z_{\gamma}$ contributes $\gamma^{\varepsilon(e,v_1)}$ to the calculation of $Syn_{R}(v_1)$, where $\varepsilon(e,v_1)$ is the function defined in \ref{signo}; and it contributes $\gamma^{\varepsilon(e,v_2)}= \gamma^{-\varepsilon(e,v_1)}$ to the calculation of $Syn_{R}(v_2)$. Thus, the error $Z_{\gamma}$ effectively contributes $\mathbf{1}$ to $\prod_{v \in V}Syn_{R}(v)$. This principle applies to every $Z$-type error on every edge, leading to the conclusion that $\prod_{v \in V}Syn_{R}(v) = \mathbf{1}$.
\end{proof}

We are going to extend the definition of the operators $A_g$ and $\widehat{B}_\gamma$ from equations \ref{vertexoperatorg} and \ref{faceoperatorgamma}, respectively. We will start with the operators $\widehat{ B}_\gamma$. Initially, the action of these operators was defined on the faces of the lattice. Now, we will define their action on any cycle $\cC$ in $\cL$. Note that for a face $p$, the action of $\widehat{B}_\gamma(p)$ on the edges of $\partial p$ depends on the counterclockwise orientation  with which $p$ is endowed. Consequently, as every cycle has two possible orientations, $\widehat{B}_\gamma(\cC)$ will depend on our chosen orientation for $\cC$. Therefore, we will first generalize the function defined in \ref{signo}
\begin{definition}
Let $\cC$ be a cycle endowed with an orientation $\sigma$ (one of the two possible orientations). Then, for the set $\{(e,\cC,\sigma)| e \in \cC\}$, we define the function
\begin{align}
    \varepsilon(e,\cC, \sigma) :=\begin{cases} ~ ~ 1 & \text{If } e \text{ is oriented in the same direction as } \sigma \text{ along } \cC.\\-1 & \text{If } e \text{ is oriented in the opposite direction as } \sigma \text{ along } \cC. \end{cases}
\end{align}
\end{definition}

Then, for each cycle $\cC$ in $\cL$ and orientation $\sigma$ of $\cC$, we define the operator
\begin{align}
    \widehat{B}_\gamma(\cC,\sigma) = \left( \bigotimes_{e \in \cC}Z^{\varepsilon(e,\cC, \sigma)}_{\gamma}\right) \otimes \left( \bigotimes_{e \in E \setminus \cC}I_e \right).
\end{align}
To simplify the notation, when the orientation of a cycle $\cC$ is clear in the context, we will write $\widehat{B}_\gamma(\cC)$ instead of $\widehat{B}_\gamma(\cC,\sigma)$, and for an edge $e \in \cC$, we will write $Z^{\varepsilon(e,\cC)}_{\gamma}$ instead of $Z^{\varepsilon(e,\cC, \sigma)}_{\gamma}$.


To extend the $A_g$ operators, we use a similar method to that used for  $\widehat{B}_\gamma$, but this time in the dual lattice. Originally, the action of $A_g$ was defined on vertices. Considering that the vertices of $\cL$ correspond to the faces of $\cL^*$, $A_g$ can effectively be interpreted as a face operator in $\cL^*$. Then, for each graph $\cC \subset \cL$ such that $\cC^* \subset \cL^*$ is a cycle and for a given orientation $\sigma$ of $\cC^*$, we define the operator 
\begin{align}
    A_g(\cC,\sigma) = \left( \bigotimes_{e \in \cC}X^{\varepsilon(e^*,\cC^*, \sigma)}_{g}\right) \otimes \left( \bigotimes_{e \in E \setminus \cC}I_e \right).
\end{align}
To simplify the notation, when the orientation of a cycle $\cC^*$ is clear in the context, we will write $A_g(\cC)$ instead of $A_g(\cC,\sigma)$, and for an edge $e \in \cC$, we will write $X^{\varepsilon(e,\cC)}_{g}$ instead of $X^{\varepsilon(e^*,\cC^*, \sigma)}_{g}$.

\begin{remark}
   It is worth mentioning that operators $\widehat{B}_\gamma(\cC,\sigma)$ and  $A_g(\cC,\sigma)$ can also be defined similarly using a path instead of a cycle. In this context, these operators are known as \emph{string operators}.
\end{remark}

\begin{lemma} \label{loopsindrome} \hfill
\begin{itemize}
    \item[$\imath$)] Let $\cC \subset \cL$ be a cycle with orientation $\sigma$, and let $\gamma \in \widehat{G}$. Consider the operator $ \widehat{B}_\gamma(\cC,\sigma)$. Then  $Syn_{\widehat{B}_\gamma(\cC,\sigma)}(v) = \textbf{1}$, for every vertex $v \in V$.
    \item[$\imath \imath$)] Let $\cC \subset \cL$ be a graph such that $\cC^*$ is a cycle with orientation $\sigma$, and let $g \in G$. Consider the operator $A_g(\cC,\sigma)$. Then $Syn_{A_g(\cC,\sigma)}(p) = 0$, for every face $p \in P$.
\end{itemize}
\end{lemma}
\begin{proof}
     $\imath$) For vertices not in $\cC$, the lemma is trivial. Now, consider a vertex $v \in \cC$. Since $\cC$ is a cycle then $v$ is incident to two edges in $\cC$, let us call them $e_1$ and $e_2$. By definition the error on $e_i$ with $i=1,2$ is $$Z^{\varepsilon(e_i,\cC)}_{\gamma} = Z_{\gamma^{\varepsilon(e_i,\cC)} }.$$
    If $e_1$ and $e_2$ are oriented in the same direction along $\cC$ then $\varepsilon(e_1,\cC)= \varepsilon(e_2,\cC)$. Besides, if they are oriented in the same direction, this implies that one edge points to $v$ and the other points away from $v$. Therefore, $\varepsilon(e_1,v)=-\varepsilon(e_2,v)$, and thus
    \begin{align*}
        Syn_R(v) & = (\gamma^{\varepsilon(e_1,\cC)})^{\varepsilon(e_1,v)}(\gamma^{\varepsilon(e_2,\cC)})^{\varepsilon(e_2,v)}\\
        & = (\gamma^{\varepsilon(e_1,\cC)})^{\varepsilon(e_1,v)}(\gamma^{\varepsilon(e_1,\cC)})^{-\varepsilon(e_1,v)}\\
        & = \gamma \gamma^{-1}\\
        & = \textbf{1}
    \end{align*}
    If $e_1$ and $e_2$ are oriented in different directions along $\cC$ the argument is similar.
    \end{proof}

\begin{lemma} \label{gammaloops} Let $R= R_XR_Z$ be an error
\begin{itemize}
    \item[$\imath$)]  Let $L$ be a bridgeless subgraph of $\G(R_Z)$ with at most one vertex $v_1 \in L$ incident to an edge of $\G(R_Z)$ not in $L$. If for all $v \in L$ with $v \neq v_1$, it holds that $Syn_R(v) = \mathbf{1}$, then for an ear-cycle cover $\{ \cC_i\}_{i=1}^{n}$ of $L$ with orientations $\{ \sigma_i\}_{i=1}^{n}$, we have
    \begin{align*}
        R_Z|_L = \prod_{i=1}^n \widehat{B}_{\gamma_i}(\cC_i, \sigma_i),
\end{align*}
for some $\gamma_i \in \widehat{G}$.
\item[$\imath \imath$)] Let $L$ a subgraph of $\G(R_X)$. Suppose $L^*$ is a bridgeless subgraph of $\G^*(R_X)$ with at most one dual vertex $p^*_1 \in L^*$ incident to an edge of $\G^*(R_X)$ not in $L^*$. If for all $p^* \in L^*$ with $p^* \neq p^*_1$, it holds that $Syn_R(p) = 0$, then for an ear-cycle cover $\{ \cC^*_i\}_{i=1}^{n}$ of $L^*$ with orientations $\{ \sigma_i\}_{i=1}^{n}$, we have
    \begin{align*}
        R_X|_L = \prod_{i=1}^n A_{g_i}(\cC_i,\sigma_i),
\end{align*}
for some $g_i \in G$.
\end{itemize}    
\end{lemma}

\begin{figure}[h!]
    \centering
  \includegraphics[scale=0.9]{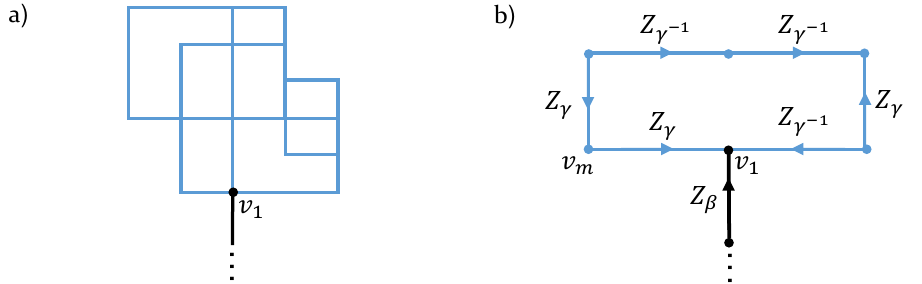}
   \caption{In a), we see in blue an example of how the bridgeless subgraph $L$, as described in Lemma \ref{gammaloops}, might look like. In b), we show the required behavior of errors on a cycle to ensure $Syn_{R}(v)= \mathbf{1}$.} 
   \label{formadeerror}
\end{figure}

Before giving the proof, it is important to clarify the objective of Lemma \ref{gammaloops}. This lemma aims to demonstrate that for a bridgeless subgraph of $Z$-type errors, as illustrated in blue in Figure \ref{formadeerror}, a), it is not necessary to know the syndrome at the vertex $v_1$. It suffices that the syndromes at all other vertices are the trivial character  in order to describe the error operators on each edge. Additionally, Lemma \ref{gammaloops} asks that at most one vertex can be incident to an edge that is not in $L$, (the edge shown in black in Figure \ref{formadeerror}). This implies that the proposition applies to both cases, the one depicted in the Figure \ref{formadeerror}, a), and when $L$ is a connected component of $\G(R_Z)$. 

\begin{proof}[ Proof Lemma \ref{gammaloops}]
  $\imath$) By definition, $L = \bigcup_{i=1}^{n} \cC_i$. We will prove this by induction on $n$. For $n=1$, $L$ is sequence of distinct edges \((e_1, e_2, \ldots, e_{m})\) such that \(\partial(e_i) = \{v_{i}, v_ {i+1}\}\), with $m+1 := 1$. The error on $e_1$ is some $Z_\gamma^{\varepsilon(e_1,L)}$, and since $Syn_{R}(v_2) = \mathbf{1}$, then the error on the edge $ e_2$ must be $Z_\gamma^{\varepsilon(e_2,L)}$. Similarly, as $Syn_{R}(v_3) = \mathbf{1}$, the error on $e_3$ is $Z_\gamma^{\varepsilon(e_3,L)}$. Continuing with this argument and considering that $Syn_{R}(v_i) = \mathbf{1}$ for all $2 \le i \le m$, it follows that the error on edge $e_i$ is $Z_\gamma^{ \varepsilon(e_i,L)}$, as seen in part b) of Figure \ref{formadeerror}. Therefore, $R_Z|_L = \widehat{B}_{\gamma}(L)$.   
  
 Now, let us assume that the proposition holds for an ear-cycle cover of $n-1$ cycles and proceed to prove that it holds for an ear-cycle cover with $n$ cycles. Take $\cC_n$ from the cover $\{ \cC_i\}_{i=1}^{n}$. By definition, there is a path or cycle $\rho_n \subset \cC_n$ that connects to $\bigcup_{i=1} ^{n-1}\cC_i$ to form $ L$. Select an edge $e \in \rho_n$; clearly, $e \not\in \cC_i$ for any $i \neq n$. The error on $e$ is some $Z_\gamma^{\varepsilon(e,\cC_n)}$. By applying the same argument as in the induction base case $n=1$, we can infer that for every edge $e' \in \rho_n$, the error is $Z_\gamma^{\varepsilon(e,\cC_n)}$. So, by applying the operator $\widehat{B}_{\gamma}(\cC_n)$ and considering the error $R'=\widehat{B}_{\gamma}(\cC_n)^{-1}R_Z$ which is composed only of $Z$-type errors, and acts as the identity on the edges in $\rho_n$. Thus $L \setminus \rho_n$ becomes a bridgeless subgraph of $\G(R')$. Let us see that $R'$ and $L \setminus \rho_n$ satisfy the hypotheses of this proposition. First, if $v_1 \not\in L\setminus \rho_n$ then choose any other vertex from $ L \setminus \rho_n$ as the new $v_1$. Secondly, as a consequence of Lemma \ref{loopsindrome}, the operator $\widehat{B}_{\gamma}(\cC_n)^{-1}$ does not affect the vertex syndrome. Therefore, for all $v \in L$ with $v \neq v_1$ it holds that  $Syn_{R'}(v) = Syn_{R}(v)= \mathbf{1}$. Moreover, $L\setminus \rho_n = \bigcup _{i=1}^{n-1} \cC_i$ with $\{ \cC_i\}_{i=1}^{n-1}$ being an ear-cycle cover of $L\setminus \rho_n$. By induction hypothesis, we have $$ R'|_L=R'|_{L \setminus \rho_n} = \prod_{i=1}^{n-1} \widehat{B}_{\gamma_i}(\cC_i) .$$
    Then $R_Z|_L = \widehat{B}_{\gamma}(L_n) R'|_L = \widehat{B}_{\gamma}(\cC_n)\prod_{i=1}^{n-1} \widehat{B}_{\gamma_i}(\cC_i)= \prod_{i=1}^{n} \widehat{B}_{\gamma_i}(\cC_i)$
\end{proof}

We now have everything we need to prove the main theorems in this section.

\begin{proof}[Proof Theorem \ref{multiloop}]
   $\imath$)  We will prove that every connected component of $\G(R_Z)$ is bridgeless. Suppose that $\G(R_Z)$ is connected but is not bridgeless. This assumption leads to two possible cases. We will show that both cases result in a contradiction, thus concluding that $\G(R_Z)$ must be bridgeless.
   
   \emph{Case 1:} There is a vertex in $\G(R_Z)$ with degree one: Let $v_1$ be the vertex that is incident to only one edge with error $Z_{\gamma}$, where $Z_{\gamma} \neq \operatorname{Id}$, so $\gamma \neq \mathbf{1}$, then $Syn_{R}(v_1) = \gamma$ or $Syn_{R}(v_1)=(\gamma)^{-1}$ depending on the orientation of the edge where the error is, in any case, $Syn_{R}(v_1) \neq \mathbf{1}$. That is a contradiction.
   
   \emph{Case 2:} Each vertex in $\G(R_Z)$ has degree more than one. By Proposition \ref{noloopforma} there exists a maximal bridgeless graph $L \subset \G(R_Z)$ featuring a unique vertex $v_1 \in L$ that is incident to a bridge $e \in \G(R_Z)$. Assuming the error on $e$ is $Z_\beta\neq \operatorname{Id}$, we aim to prove that $Syn_{R}(v_1) \neq \mathbf{1}$, which would be a contradiction. Consider $L = \bigcup_{i=1}^{n} \cC_i$ with $\{ \cC_i\}_{i=1}^{n}$ ear-cycle cover of $L$. By fixing a set of orientations $\{ \sigma_i\}_{i=1}^{n}$ for $\{ \cC_i\}_{i=1}^{n}$, Lemma \ref{gammaloops} leads to the equation  $$R_Z|_L = \prod_{i=1}^{n} \widehat{B}_{\gamma_i}(\cC_i).$$
   However, according to Lemma \ref{loopsindrome} operators of the form $\widehat{B}_{\gamma_i}(\cC_i)$ have no effect on the vertex syndrome. Then $Syn_{R}(v_1) = Z_\beta^{\varepsilon(e,v_1)} \neq \mathbf{1}$, which presents a contradiction.
 \end{proof}

\begin{proof}[Proof Theorem \ref{multiloop contractil}]
    $\imath$) To prove that $R_Z$ acts as the identity on $V_{gs}$, we aim to prove that $R_Z= \prod_{p\in P}\widehat{B}_{\gamma_p}(p)$. According to Proposition \ref{bprimas}, for every vector $\ket{\psi} \in V_{gs}$ and every face $p \in P$, it holds that $\widehat{B}_{\gamma_p}(p)\ket{\psi} = \ket{\psi}$. Because $Syn_{R}(v) = \mathbf{1}$ for all $v \in V$, by Theorem \ref{multiloop} $\G(R_Z)$ is bridgeless, so we can select an ear-cycle cover $\G(R_Z) = \bigcup_{i=1}^{n} \cC_i$, by hypothesis each $\cC_i$ is contractible, allowing us to assign a counterclockwise orientation to each. The proof then proceeds by induction on the number of cycles in the ear-cycle cover. Let us assume that $\G(R_Z) = \cC_1$ is a single cycle. Then, by Lemma \ref{gammaloops}, $R_Z=  \widehat{B}_{\gamma}(\cC_1)$. Since the cycle $\G(R_Z)= \cC_1$ is contractible, its interior is composed of a finite union of faces, lets $\{ p_1, p_2, \dots, p_r\}$ said faces. Hence, we have $R_Z = \prod_{i=1}^r \widehat{B}_\gamma(p_i)$. Now, suppose this statement holds for an ear-cycle cover consisting of  $n-1$ cycles. If $\G(R_Z) = \bigcup_{i=1}^{n} \cC_i$, then according to Lemma \ref{gammaloops}, $$R_Z = \prod_{i=1}^{n} \widehat{B}_{\gamma_i}(\cC_i).$$
    Take the faces $\{ p_1,p_2, \dots, p_r\}$ that lie in $L_n$ and consider the operator $\prod_{i=1}^r\widehat{B}_\gamma(p_i)$. Then $\widehat{B}_{\gamma_n}(\cC_n) = \prod_{i=1}^r\widehat{B}_\gamma(p_i)$. Thus, $R_Z = R'\prod_{i=1}^r\widehat{B}_{\gamma}(p_i)$, where $R'$ is a $Z$-type error such that $Syn_{R'}(v) =  \mathbf{1}$ for all $v \in V$, and $\G(R')$ is contractible bridgeless graph consisting of the cycles  $\{\cC_1, \dots, \cC_{n-1} \}$.  Following the induction hypothesis, we get $R' = \prod_{p\in P}\widehat{B}_{\gamma_p}(p)$ and so $R_Z =  \prod_{p\in P}\widehat{B}_{\gamma_p}(p) \prod_{i=1}^r\widehat{B}_{\gamma_n}(p_i)$, which concludes the proof.
 \end{proof}

\section{Theoretical set of correctable errors}\label{numeroteorico}

By Proposition \ref{base}, we have that a basis for all errors over $V_{gs}$ is
$$\cR = \left\{ \bigotimes_{e \in E}X_{ g_e}Z_{\gamma_e} :  g_e \in G,\gamma_e \in \widehat{G} \right\},$$ we would like to know which subsets of $\cR$ are correctable. For this, we have the Knill-Laflamme Theorem \cite{knill-laflamme}, which gives us the necessary and sufficient conditions for a code to be error correctable for a specific set of errors.


 \begin{theorem}[Knill-Laflamme] \hfill \label{teo: knill-laflame}\\
Let $W$ be a quantum code,
and let $\cP$ be the projector onto $W$. Suppose $R$ is a quantum operator with
error operators $\{R_i\}$. A necessary and sufficient condition for the existence
of an error-correction operator  correcting $R$ on $W$ is that
\begin{align} \label{condicion1}
   \cP R_i^{\dagger}R_j\cP= \alpha_{ij}\cP 
\end{align}
for some Hermitian matrix $\alpha$ of complex numbers.     
\qed
 \end{theorem}

Being $V_{gs}$ our code, the projector onto $V_{gs}$ is shown  in equation \ref{proyector}.


 \begin{theorem}
  \label{teoricamente}
Let $G$ be an abelian group, and $\Sigma$ be a compact oriented surface. Consider an oriented lattice $\mathcal{L}$ on $\Sigma$, such that $2n+1$ is the minimum number of edges required to form a non-contractible cycle in $\Sigma$. Then $V_{gs}$ is a quantum error-correcting code for the set,
\begin{align}
    \mathfrak{R} = \{ R~|~R \text{ is an error that affects at most } n \text{ edges} \}.
\end{align}
\end{theorem}
 \begin{proof}
  Let $\cP$ be the projector onto $V_{gs}$. Let 
$R_i,R_j \in \mathfrak{R}$, consider $R_i^{\dagger}R_j$. By Propositions \ref{long}, for each site $(v,p)$, we have two possibilities:
\begin{align*}
         P_{0,\mathbf{1}}(v,p)(R_i^{\dagger}R_j)\cP = 0 && \text{or} && P_{0,\mathbf{1}}(v,p)(R_i^{\dagger}R_j ) \cP= (R_i^{\dagger}R_j) \cP
     \end{align*}
If for some site $P_{0,\mathbf{1}}(v,p)(R_i^{\dagger}R_j)\cP = 0$, then it follows that $\cP(R_i^{\dagger}R_j)\cP = 0$. Conversely, if for every site $P_{0,\mathbf{1}}(v,p)(R_i^{\dagger}R_j)\cP = (R_i^{\dagger}R_j) \cP$, then for every vertex $v$, the vertex syndrome $Syn_{R_i^{\dagger}R_j}(v) = \mathbf{1}$ and for every face $p$, the face syndrome $Syn_{R_i^{\dagger}R_j}(p) = 0$. According to Theorem \ref{multiloop}, $\G((R_i^{\dagger}R_j)_Z)$ and $\G^*((R_i^{\dagger}R_j)_X)$ are bridgeless graphs. But since $R_i^{\dagger}$ and $R_j$ each act on at most $n$ edges, $R_i^{\dagger}R_j$ acts on at most $2n$ edges. Hence, $\G((R_i^{\dagger}R_j)_Z)$ and $\G^*((R_i^{\dagger}R_j)_X)$ must be contractible, then by Theorem \ref{multiloop contractil}, we have that $(R_i^{\dagger}R_j) \cP = \cP$, which also means $\cP (R_i^{\dagger}R_j ) \cP = \cP$. In conclusion, by Theorem \ref{teo: knill-laflame}, $V_{gs}$ is a quantum error-correcting code for the set of errors $\mathfrak{R}$.
 \end{proof}





Theorem \ref{teoricamente} says that, if $2n$ edges are insufficient to form a non-contractible cycle in $\Sigma$, then $V_{gs}$ is capable of correcting errors affecting $n$ edges. This establishes a theoretical set of correctable errors. However, the theoretical correctability of this set does not guarantee its practical feasibility. In fact, the construction of correction operators for the set of errors outlined in Theorem \ref{teoricamente}, which ensure the resultant graph is contractile, is, in the general case, an NP-complete problem. These issues will be further discussed in Section \ref{apendice}. Alternatively, Section \ref{seccioncorrecion} will introduce a distinct set of correctable errors that, despite being smaller than the one proposed in Theorem \ref{teoricamente}, can be corrected in polynomial time.

\section{Error correction for abelian groups} \label{seccioncorrecion}

As mentioned in Section \ref{secciondeteccion}, a basic error has the form $R = \bigotimes_{e \in E}X_{ g_e}Z_{\gamma_e}$, and $R$ can always be rewritten, up to a phase, as $R= R_XR_Z$, where $R_X = \bigotimes_{e \in E}X_{ g_e}$ and $R_Z = \bigotimes_{e \in E}Z_{\gamma_e}$. Given that $R_X$ and $R_Z$ can be corrected separately, this section focuses on developing a correction operator for $Z$-type errors. The methodology for correcting $X$-type errors, which involves working with the dual lattice, is analogous but not covered here.

The construction of a correction operator $C$, for a $Z$-type error $R$, depends on two essential requirements. First, for each vertex $v \in V$, we require $Syn_C(v) = Syn_R(v)^{-1}$. This condition ensures $Syn_{CR}(v) = \mathbf{1}$, which, according to Theorem \ref{multiloop}, implies that $\G(CR)$ is a bridgeless graph. Second, it is necessary for this bridgeless graph to be contractible. By Theorem \ref{multiloop contractil}, contractibility of $\G(CR)$  guarantees that $CR$ fixes the ground states,  and then $C$ corrects $R$. The preliminary step in constructing $C$ involves identifying the edges on which the correction will act. These edges are chosen using Algorithm \ref{redes}, which is a variant of the Hard Decision Renormalization Group (HDRG) decoders employed in various topological quantum error correction algorithms, as in \cite{clu5,clu4, clu3}.

Algorithm \ref{redes} incorporates Dijkstra's algorithm to determine the minimum weight paths between vertices in a weighted graph. Starting from a selected vertex, the algorithm find the minimum weight path to every other vertex in the graph. When dealing with a subset of vertices $\{v_1, \dots, v_r \}$, Dijkstra's algorithm is applied individually to each $v_i \in \{v_1, \dots, v_r \}$, computing the minimum weight paths between all vertex pairs in the subset. Thus, the algorithm is executed $r$ times, where $r$ denotes the number of vertices in the subset.

\begin{definition}
  Let $ \cL$ be an oriented lattice.
  \begin{itemize}
  \item[$\bullet$] A \emph{cluster}\footnote[1]{In graph theory, the term \emph{cluster} carries a different definition. However, this term was chosen to describe connected graphs, as it has been commonly used in some papers about error correction in quantum codes.} of $\cL$ is a connected subgraph of $\cL$.
  \item[$\bullet$] Let $N$ be a cluster. A \emph{subcluster} of $N$ is a cluster $N'$ such that $N' \subset N$.
  
  \item[$\bullet$] Let $ \rho$ be  a  path. The length of $\rho$, denoted by $l(\rho)$, is the number of edges in $\rho$.    
  \item[$\bullet$] Let $v_1, v_2$ be two vertices of $\cL$. We define the distance between $v_1$ and $v_2$, $d(v_1,v_2)$, as the minimum length of a  path that connects $v_1 $ and $v_2$. 
  
  \end{itemize}
\end{definition}

\begin{definition}
Let $N$ be a cluster, and \(x: W \to \widehat{G}\) be a function from \(W\), a subset of the vertices of \(N\), to \(\widehat{G}\). $N$ is called \emph{neutral} for the function $x$ if $\prod_{v \in W} x(v) = \mathbf{1}$.
\end{definition}

For a cluster $N$ and a function $x: W \to \widehat{G}$. In contexts where it is clear which function $x$ we are referring to, we will simply say that $N$ is neutral, instead of saying $N$ is neutral for the function $x$.

Now, we proceed to determine the edges used by a correction operator. Consider a state $\ket{\psi} \in V_{gs}$ and a $Z$-type error operator $R$ such that $R\ket{\psi} = \ket{\psi'}$. For each site $(v,p)$ where $P_{0,\mathbf{1}}(v,p)\ket{\psi'} = 0$, we identify the operator $P_{g,\gamma}(v,p)$ that ensures $P_{g,\gamma}(v,p)\ket{\psi'} = \ket{\psi'}$. This process leads to the determination of the labels $Syn_{R}(v) = \gamma \neq \mathbf{1}$. Next, we construct clusters by connecting vertices that are labeled.


\begin{algoritmo} \label{redes}
\textbf{Algorithm for the construction of the correction graph $\G(C)$:}\\
Input: Let $W \subset V$, and $x: W \to \widehat{G}$ be a function such that for each $v \in W$, $x(v) \neq \mathbf{1}$, and $\prod_{v \in W}x(v) = \mathbf{1}$.\\
Output: A graph $\G(C) \subset \cL$ with connected components $N_1, \dots, N_m$, each of which is a neutral cluster, and for each vertex $v \in W$, there is a $N_k$ such that $v \in N_k$.\\

Initialization: \begin{itemize}
        \item Initially $\G(C) = W$.
        \item Consider the complete graph $\T$ with vertex set $W $.
        \item Associate a weight to each edge of $\T$. Initially, for $v_i,v_j \in W$, the weight $w(v_i,v_j)$ is set to $d(v_i,v_j)$.
        \item Take $W' \subset W$, initially setting $W' = W$.\\
    \end{itemize}
    
\begin{itemize}
    \item[1)] Apply Dijkstra’s algorithm on $\T$ to find the minimum weight paths between any pair of vertices in $W'$, and let $\tau_{ij} \subset \T$ be a minimum weight path between $v_i,v_j \in W'$, with weigh $w(\tau_{ij})$.
    
    \item[2)] Take $t= \min \{w(\tau_{ij}): v_i,v_j \in W', i \neq j \}$. If for $v_r,v_s \in W'$, $t= w(\tau_{rs})$, then $$\tau_{rs} = (v_r=v_1,v_2,v_3, \dots , v_n = v_s) \subset \T,$$ where $\sum_{i=1}^{n-1}w(v_i,v_{i+1}) = t$. For each $1\le i \le n-1$, if $w(v_i,v_{i+1}) \neq 0$ construct a path $\rho \subset \cL$ of length $w(v_i,v_{i+1})$ that connects $v_i$ with $v_{i+1}$ in $\cL$.

    \item[3)] Add every path $\rho$ constructed in the previous step to the graph $\G(C)$.

    \item[4)] Take the connected components $N_1, \dots, N_n$ of the graph $\G(C)$. Consider the partition of $W$ in the sets $W_k = W \cap N_k$.
    \begin{itemize}
       \item[4.1)] Update weights. If $v_i, v_j \in W$ belong to the same connected component $N_k$, set $w(v_i,v_j) = 0$; otherwise, keep $w(v_i,v_j) = d(v_i,v_j)$. 
       \item[4.2)] Update $W'$. For each $W_k$ pick a vertex $v_k \in W_k$, thus redefining $W' = \{ v_k | N_k \text{ non-neutral}\}$
    \end{itemize}
    
    \item[5)] If $W' \neq \varnothing$, repeat from step $1$. If $W' = \varnothing$, the algorithm ends. 
    
\end{itemize}
\end{algoritmo}

\begin{proposition}
    The Algorithm for the construction of the correction graph \ref{redes} runs in polynomial time.
\end{proposition}
\begin{proof}
Let $W = \{ v_1, \dots, v_r\}$ be the initial set of vertices for Algorithm \ref{redes}. Each iteration of steps 1 to 5 involves running Dijkstra's algorithm as many times as the cardinal of $W'$, denoted by $|W'|$. Given that Dijkstra's algorithm has a time complexity of $O(|W|^2)$, the execution of steps 1 to 5 collectively has a time complexity of $O(|W|^3)$. Since each iteration of steps 1 to 5 reduces the size of $W'$ by at least one, these steps are repeated at most $|W|$ times. Consequently, the overall time complexity of the Algorithm for the construction of the correction graph is $O(|W|^4)$.

\end{proof}

Now, we will see an example of how the Algorithm \ref{redes} works

\begin{example} \label{ejemplored}

\begin{figure}[h!]
    \centering
  \includegraphics[scale=1]{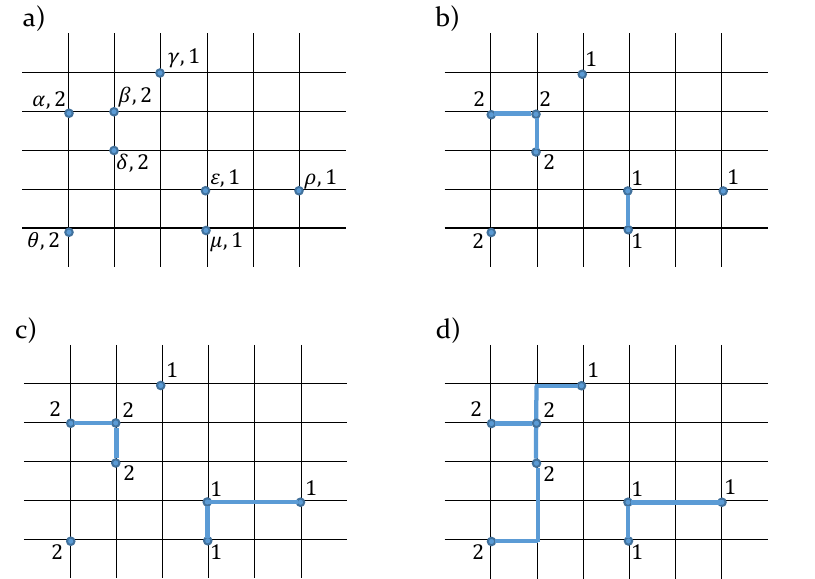}
   \caption{Example of creating clusters using Algorithm \ref{redes}.}
   \label{creacioncuerda}
\end{figure}
Let $G= \mathbb{Z}_3$; we know that $\widehat{\mathbb{Z}_3} \simeq \mathbb{Z}_3$. Consider the lattice in Figure \ref{creacioncuerda}. In part a), we have a set of vertices $W=\{\alpha,\beta,\gamma,\delta, \varepsilon, \theta, \mu, \rho \}$, each labeled with elements of $\widehat{\mathbb{Z}_3}$. Consider the complete graph $\T$ with vertex set $W$. Initially, for any $v_i, v_j \in W$, the weight of the edge of $\T$ connecting $v_i$ with $v_j$ is determined by the distance between $v_i$ and $v_j$ in the lattice $\mathcal{L}$, i.e., $w(v_i, v_j) = d(v_i, v_j)$. Also, initially, we have $W' = W$.

Thus, the first paths that Algorithm \ref{redes} creates in the lattice $\mathcal{L}$ are shown in Figure \ref{creacioncuerda}, b), connecting vertices of $W$  that are at the minimum distance from each other. In this case, the shortest distance is one. Vertices $\alpha, \beta,$ and $\delta$ are connected, forming a neutral cluster as the sum of their labels $2+2+2 \equiv 0$. On the other hand, $\varepsilon$ is connected to $\mu$, where the sum of their labels $1+1 \not\equiv 0$. Thus, we have five connected components, and the five sets of vertices $W_k$ will be
\begin{align*}
    W_1 = \{\alpha, \beta, \delta \} && W_2 = \{\gamma \} && W_3 = \{ \theta\} && W_4 = \{ \mu, \varepsilon\} && W_5 = \{\rho  \}
\end{align*}
The set $W'$ is updated into $W' = \{ \gamma,\theta, \mu, \rho \}$, which has one vertex from each non-neutral cluster. We also update the weights of the edges of the graph $\T$. For instance, now $w(\alpha,\beta) = 0$, since they belong to the same connected component. For vertices in different connected components, such as $\gamma$ and $\varepsilon$, we still have $w(\gamma,\varepsilon) =d(\gamma,\varepsilon) = 4$. In Figure \ref{creacioncuerda}, c), we have the following step of the Algorithm, which connects the vertices $\mu$ and $\rho$, since in the graph $\mathcal{T}$, the path with the minimum weight connecting $\mu$ and $\rho$ is the path $(\mu, \varepsilon, \rho)$ with a weight of $w(\mu, \varepsilon) + w(\varepsilon, \rho) = 0 + 2 = 2$. As $w(\varepsilon, \rho) = 2 \neq 0$, we create a path of length $2$ in $\mathcal{L}$ that connects $\varepsilon$ and $\rho$. This effectively connects $\mu$ and $\rho$ by taking a shortcut through $\varepsilon$, taking advantage of the fact that $w(\mu,\varepsilon)=0$, indicating that $\mu$ and $\varepsilon$ were already connected in $\mathcal{L}$.  Now we have four connected components where
\begin{align*}
    W_1 = \{\alpha, \beta, \delta \} && W_2 = \{\gamma \} && W_3 = \{ \theta \} && W_4 = \{ \mu, \varepsilon, \rho\}
\end{align*}
The set $W'$ changes to $W' = \{ \gamma,\theta \}$. Therefore, the last step of the Algorithm, seen d) of Figure \ref{creacioncuerda}, is to connect $\gamma$ and $\theta$. Note that the path with the minimum weight between them in $\T$ is $(\theta, \delta, \beta, \gamma)$ with a weight of $w(\theta, \delta) + w(\delta, \beta)+ w(\beta, \gamma)= 3 +0 +2 =5$. It becomes evident that the intention of Algorithm \ref{redes} in redefining the weights is to use shortcuts through paths already created to minimize the use of edges when connecting vertices. Thus, Algorithm \ref{redes}, when using shortcuts, can create several paths to connect two vertices. For example, in d), $\theta$ and $\gamma$ were connected using two  paths, as a shortcut along the path between $\delta$ and $\beta$ was taken, with one path from  $\theta$ to $\delta$ of length $w(\theta, \delta)=3$, and another from $\beta$ to $\gamma$ of length $w(\beta, \gamma)= 2$. As a final result of Algorithm \ref{redes}, we obtain a graph with two neutral clusters.
\end{example}

The following proposition demonstrates how, after selecting the graph for the correction operator, we can determine the operator acting on the edges of this graph, which will serve as the correction operator.

\begin{proposition} \label{existe}
Let $N$ be a cluster and $\{ v_1, \dots, v_r\}$ be the set of vertices that belong to $N$. To each vertex $v_k$, we assign an element $x(v_k) \in \widehat{G}$ such that $\prod_{k=1}^rx(v_k) = \mathbf{1}$, then there exists a $Z$-type operator $T$ acting only on the edges of the cluster $N$ such that for every vertex $v_k$, $Syn_T(v_k)=x(v_k)$.
\end{proposition}
\begin{proof}
   Select a spanning tree $N'$ of $N$. This tree $N'$ serves as a subcluster of $N$. To the edges $e \in N \setminus N'$, we will apply the operator $Z_{\mathbf{1}}$. Considering $N'$ is a tree, it follows that certain vertices are incident to only one edge; suppose that $v_1$ is one of these vertices, where $v_1$ is incident to $e_1 \in N'$. So the operator that we will apply to $e_1$ will be $Z_{x(v_1)}^{\varepsilon(e_1,v_1)}$, with $\varepsilon(e_1,v_1)$ the function defined in \ref{signo}.
    After defining the operator for edge $e_1\in N'$, we need to look at the other vertex incident to $e_1$, denoted $v_2$. We have to update $x(v_2)$ to be $x(v_2) := x(v_2)x(v_1)$. The same steps we used for $e_1$ can be applied to get the operator for another edge, say $e_2$, in the tree $N' \setminus \{ e_1\}$. Then, we repeat this for $N' \setminus \{ e_1, e_2\}$ and so on. Since $N'$ is a tree with $r$ vertices and $r-1$ edges, we will do this $r-1$ times to find the operator we must apply to each edge. So, for every edge $e$ in $N$, we have found the operator $Z_{\gamma_e}$ that we apply. Then by considering the operator $$T = \left( \bigotimes_{e \in N}Z_{\gamma_e}\right) \otimes \left( \bigotimes_{e \notin N}I_e \right),$$ we make sure that $Syn_T(v_k) = x(v_k)$ for all $1 \le k \le r$.
\end{proof}

We now have everything we need to describe the process of correcting $Z$-type errors.\\
 
\fbox{
\begin{minipage}[c][1.1\height]
[c]{0.92 \textwidth}
\begin{algoritmo} \label{algoritm}
\begin{center}
   \textbf{Error Correction Algorithm for a $Z$-type error}
\end{center}
\end{algoritmo}
Consider the group $G$ and an oriented lattice $ \cL$ on a compact oriented surface $\Sigma$. Let $\ket{\psi} \in V_{gs}$ and $R$ be a $Z$-type error operator such that $R\ket{\psi} = \ket{\psi'}$, where $R$ affects at most $n$ edges. If  the minimum number of edges needed to form a non-contractible cycle in $\Sigma$ is more than $ \left\lfloor n\left( 
\frac{2+ \log_2(n)}{2}\right) +1 \right\rfloor$. Then a correction operator $C$ is constructed like this:
\begin{itemize}
    \item[1)] For each site $(v,p)$ we measure the state $\ket{\psi'}$ with respect to the set of the projectors $\{P_{g,\gamma}(v,p)\ | \ g \in G, \gamma \in \hat{G}\}$. The outcome of the measurement tells us for which $(g,\gamma)$ we have $P_{g,\gamma}(v,p)\ket{\psi'} = \ket{\psi'}$.
    \item[2)] Select the sites $(v,p)$ for which $P_{g,\gamma}(v,p)\ket{\psi'} = \ket{\psi'}$ with $P_{g,\gamma} \neq P_{0,\textbf{1}}$. 
    \item[3)] For sites $(v,p)$ of the previous item, we label the vertex $v$ with $Syn_{R}(v)^{-1}$.
    \item[4)] Establish the edges on which the correction operator will act. Take $\{ v_1, \dots, v_r\}$ the set of labeled vertices and build $\G(C)$ as indicated in Algorithm \ref{redes}, with $x(v_i) = Syn_{R}(v_i)^{-1} $.
    \item[5)] Step 4 result in a graph $\G(C)$ with disjoint neutral clusters $N_1, \dots, N_m$, such that for every vertex $v_s \in \{ v_1, \dots, v_r\}$, there is a cluster $N_i$ such that $v_s \in N_i$.
    \item[6)] For each cluster $N_i$, we construct the operator $T_i$ of Proposition \ref{existe}.
    \item[7)] Take operator $C = \prod_{i=1}^m T_i$, which applies operators of the form $Z_{\gamma}$ only on the edges in the clusters $N_1, \dots, N_m $, and for every vertex $v$, $Syn_C(v) = Syn_{R}(v)^{-1}$.
    \item[8)] Operator $C$ corrects the error $R$. It only remains to apply $C$ to $\ket{\psi'}$.
    \end{itemize}
\end{minipage}}\\\\

We have described the error correction algorithm for $Z$-type errors, although it is not evident why it is necessary that  the minimum number of edges needed to form a non-contractible cycle in $\Sigma$ be more than $ \left\lfloor n\left( 
\frac{2+ \log_2(n)}{2}\right)+1 \right\rfloor$. This guarantees two things, firstly,  that $R$ does not form a non-contractible graph, and thus it will be detected. And secondly, that the $\G(CR)$ is bridgeless graph; as shown in Proposition \ref{dificil}.\\



\begin{definition} \hfill
\begin{itemize}
    \item A path $\rho$, is an error path if $\rho \subset \G(R) \subset \G(CR)$.
        \item A path $\rho$, is a correction path if $\rho \subset \G(C) \subset \G(CR)$.
\end{itemize}
\end{definition}

\begin{definition} \label{deflevel} The concepts defined here where adapted form \cite{clu4}.
    \begin{itemize}
        \item Let $N$ a cluster of $\G(CR)$, the \emph{width} of $N$, will be the number of edges in $N$ that have errors, we denote $h(N)$.
        \item A \emph{level-0 cluster} of $\G(CR)$ is a cluster of $\G(R)$.
        \item A \emph{level-1 cluster} of $\G(CR)$ is a cluster that contains a collection of at least two level-0 clusters $\{ N_i \}$, such that for any pair $N_i, N_j$ that are connected by a correction path $\rho$, we have $l(\rho) \le \operatorname{min} \{h(N_i), h(N_j)\}$.
         \item A \emph{level-m cluster} of $\G(CR)$ is a cluster that contains a collection of at least two clusters $\{ N_i \}$, where each $N_i$ is a level-$k_i$ cluster with $k_i < m$, and at least one is a level-$(m-1)$ cluster. Furthermore, for any pair $N_i, N_j$ that are connected by a correction  path $\rho$, we have $l(\rho) \le \operatorname{min} \{h(N_i), h(N_j)\} $.
    \end{itemize}
\end{definition}

\begin{proposition}\label{dificil}
 Let $R$ be a $Z$-type error that affects $n$ edges, and $C$ be the correction operator built-in Algorithm \ref{algoritm}. Then every cycle of $\G(CR)$ will have a maximum length of $ f(n) = \left\lfloor n\left( 
\frac{2+ \log_2(n)}{2}\right) +1\right\rfloor$.
\end{proposition}
\begin{proof}
    Let $\cC$ be a cycle of $\G(CR)$. The creation of this cycle occurs because at some point Algorithm \ref{redes} connects two vertices $v_1$ and $v_2$ that were previously connected by a  path $\rho \subset \G(CR)$. When the Algorithm connects $v_1$ with $v_2$, it is done through a  path of length $w(v_1,v_2)$. Therefore, the length of $\cC$ satisfies that
    $$l(\cC) = l(\rho)+ w(v_1,v_2).$$
    Furthermore, the fact that $v_1$ and $v_2$ belonged to $\rho$ tells us that before Algorithm \ref{redes} connected them, they were already part of the same level-$m$ cluster in $\G(CR)$. We will proceed by induction on $m$. For the level-$0$ cluster the  path $\rho$, that connects $v_1$ with $v_2$, is only composed of errors then $l(\rho) \le n$ and then we will have a cycle of length at most $2n$ edges, where clearly $2n \le f(n)$.

    Suppose that the proposition holds when $\rho$ is contained in a level-$k$ cluster with $k<m$ and let us now prove the case of a level-$m$ cluster. Note that $\rho$ is a succession of correction  paths  and error  paths, for the Algorithm \ref{redes} the weight $w(v_1,v_2)$ is bounded by the number of errors in $\rho$, that is, $w(v_1 ,v_2) \le h(\rho)$, this is because the correction  paths  do not add to the distance as they are shortcuts.    Moreover, by the definition of a level-$m$ cluster, the path $\rho$ can also be viewed as a succession of correction paths $\{\rho_j\}_{j=1}^J$ and paths $\{\rho_i\}_{i=1}^I$ contained in a level-$k$ cluster, where $k < m$. Let $\{N_i\}_{i=1}^I$ be the set of clusters containing the paths $\{\rho_i\}_{i=1}^I$, such that $\rho_i = \rho \cap N_i$. The following shows what a section of $\rho$ looks like:
    
    \begin{center}
\tikzset{every picture/.style={line width=0.75pt}} 
\begin{tikzpicture}[x=0.50pt,y=0.50pt,yscale=-1,xscale=1]
\draw    (100,112) -- (213,112) ;
\draw    (326,112) -- (439,112) ;
\draw [color={rgb, 255:red, 0; green, 0; blue, 0 }  ,draw opacity=1 ]   (213,112) -- (326,112) ;
\draw   (213,112) -- (234.19,77) -- (304.81,77) -- (326,112) -- (304.81,147) -- (234.19,147) -- cycle ;

\draw (155,84.4) node [anchor=north west][inner sep=0.75pt]    {$\rho_{j}$};
\draw (361,84.4) node [anchor=north west][inner sep=0.75pt]    {$\rho_{j+1}$};
\draw (261,86.4) node [anchor=north west][inner sep=0.75pt]    {$\rho_{i}$};
\draw (256,51.4) node [anchor=north west][inner sep=0.75pt]    {$N_{i}$};
\end{tikzpicture}
    
\end{center}
    
    Then $w(v_1,v_2)$ is bounded by the number of errors in $\{ \rho_i\}_{i=1}^I$. Furthermore, for each $\rho_i \in \{ \rho_i\}_{i=1}^I$, consider its two endpoints $u_i$ and $u'_i$. So 
    $w(v_1,v_2) = \sum_{i=1}^I w(u_i,u'_i).$
    In addition, by the induction hypothesis, the cycle formed with $\rho_i$ by connecting $u_i$ and $u'_i$ with a  path of length $w(u_i,u'_i)$ must have a length of at most $f(h(N_i))$. Therefore, $l(\rho_i) \le f(h(N_i))- w(u_i,u'_i)$. So,
    $$l(\cC) = \sum_{j=1}^J l(\rho_j) +\sum_{i=1}^I l(\rho_i)+ w(v_1,v_2) \le \sum_{j=1}^J l(\rho_j) +\sum_{i=1}^I f(h(N_i)) $$    
    We will now see how the length of the correction  paths  is related to some errors. Let $\rho_j \in \{ \rho_j\}_{j=1}^J$. Let us consider the vertex $u_{j}$ at the right end of $\rho_j$. This vertex belongs to a level-$k$ cluster, which we denote by $N_{j}$. It is important to note that $N_{j}$ may or may not be part of the set $\{ N_i\}_{i=1}^I$. Both cases are illustrated below

    \begin{center}
        \tikzset{every picture/.style={line width=0.75pt}} 

\begin{tikzpicture}[x=0.5pt,y=0.5pt,yscale=-1,xscale=1]

\draw    (11,111) -- (92,111) ;
\draw    (205,111) -- (293,111) ;
\draw   (92,111) -- (113.19,76) -- (183.81,76) -- (205,111) -- (183.81,146) -- (113.19,146) -- cycle ;
\draw    (406,162) -- (519,162) ;
\draw    (519,162) -- (632,162) ;
\draw   (519,162) -- (484.01,140.79) -- (484.06,70.16) -- (519.08,49) -- (554.06,70.21) -- (554.01,140.84) -- cycle ;
\draw  [fill={rgb, 255:red, 208; green, 2; blue, 27 }  ,fill opacity=1 ] (88.8,110.95) .. controls (88.82,109.18) and (90.28,107.77) .. (92.05,107.8) .. controls (93.82,107.82) and (95.23,109.28) .. (95.2,111.05) .. controls (95.18,112.82) and (93.72,114.23) .. (91.95,114.2) .. controls (90.18,114.18) and (88.77,112.72) .. (88.8,110.95) -- cycle ;
\draw  [fill={rgb, 255:red, 208; green, 2; blue, 27 }  ,fill opacity=1 ] (515.8,161.95) .. controls (515.82,160.18) and (517.28,158.77) .. (519.05,158.8) .. controls (520.82,158.82) and (522.23,160.28) .. (522.2,162.05) .. controls (522.18,163.82) and (520.72,165.23) .. (518.95,165.2) .. controls (517.18,165.18) and (515.77,163.72) .. (515.8,161.95) -- cycle ;

\draw (37,85.4) node [anchor=north west][inner sep=0.75pt]    {$\rho_{j}$};
\draw (239,85.4) node [anchor=north west][inner sep=0.75pt]    {$\rho_{j+1}$};
\draw (134,98.4) node [anchor=north west][inner sep=0.75pt]    {$N_{j}$};
\draw (445,136.4) node [anchor=north west][inner sep=0.75pt]    {$\rho_{j}$};
\draw (578,136.4) node [anchor=north west][inner sep=0.75pt]    {$\rho_{j+1}$};
\draw (504,96.4) node [anchor=north west][inner sep=0.75pt]    {$N_{j}$};
\draw (71,118.4) node [anchor=north west][inner sep=0.75pt]    {$u_{j}$};
\draw (510,168.4) node [anchor=north west][inner sep=0.75pt]    {$u_{j}$};
\end{tikzpicture}

    \end{center}
    
    Since we are in a level-$m$ cluster, we have that $l(\rho_j) \le h(N_{j})$. However, the length of $\rho_j$ is also bounded by the cluster adjacent to $\rho_j$ through its left vertex, which is $N_{j-1}$. Thus, to maximize the length of $\rho_j$ with respect to the number of errors, it is best for each cluster to have the same width and distribute the errors uniformly. That is, for $J$ correction  paths , the errors are distributed in $J+1$ clusters, each with a width of $\frac{n}{J+1}$. Therefore, $h(N_j)=\frac{n}{J+1}$ for all $j \in J$. Then, the $J$ correction  paths  together will have a length of at most $\frac{Jn}{J+1}$. Therefore, we can affirm that,
    $$\sum_{j=1}^J l(\rho_j)  \le  \frac{Jn}{J+1}$$    
    But also, notice that    
    $$\sum_{i=1}^I f(h(N_i)) + \sum_{j=1}^J l(\rho_j) \le \sum_{j=0}^J f(h(N_j)) + \frac{Jn}{J+1} - J$$   
   This is because for every $N_i \in \{N_i\}_{i=1}^I$, we also have $N_i \in \{N_j\}_{j=0}^J$. In other words, each $N_i$ equals some $N_j$ for a certain $j \in J$. However, for each $N_j$, there are two possibilities. One possibility is that $N_{j} \not\in \{ N_i\}_{i=1}^I$, which implies that $\sum_{i=1 }^I f(h(N_i))$ will be less than $\sum_{j=0}^J f(h(N_j))$ by at least one. The other possibility is that $N_{j} \in \{ N_i\}_{i=1}^I$. In this case, since $l(\rho_j) \le h(N_j)$, if $l(\rho_j) < h(N_{j})$, then $l(\rho_j) \le h(N_{j})-1$. Alternatively, if $l(\rho_j) = h(N_{j})$, then Algorithm \ref{redes} would have created $\rho_j$ and also created a shortcut through the cluster $N_{j} =N_i$. As a result, the weight $w(u_i,u'_i)$, which depends on $N_i$, would no longer add to $w(v_1,v_2)$, thereby decreasing it by at least one. Therefore, in either scenario, for each $N_j$, we can subtract one from the inequality. So,
    $$ l(\cC) \le \sum_{j=0}^J f(h(N_j)) + \frac{Jn}{J+1} - J$$
   Let us remember that $h(N_j)=\frac{n}{J+1}$ and that $J$ is an integer greater than one. Using simple calculus, we can see that the maximum of the expression on the right occurs when $J= 1$. Then
   \begin{align*}
       l(\cC) & \le \sum_{j=0}^1 f\left(\frac{n}{2}\right) + \frac{n}{2} - 1\\
       & \le 2 \left( \frac{n}{2}\left( 
\frac{2+ \log_2(\frac{n}{2})}{2}\right) +1 \right) + \frac{n}{2} - 1 \\
       & = n\left( 
\frac{2+ \log_2(n)}{2}\right) +1
   \end{align*}
\end{proof}

\begin{theorem} \label{teorema}
Let $G$ be an abelian group, and $\Sigma$ be a compact oriented surface. Consider an oriented lattice $\mathcal{L}$ on $\Sigma$, such that $ \left\lfloor n\left( 
\frac{2+ \log_2(n)}{2}\right) +1\right\rfloor$ is the minimum number of edges required to form a non-contractible cycle in $\Sigma$. Then $V_{gs}$ can correct any $Z$-type error that affects at most $n$ edges.
\end{theorem}
\begin{proof}
    Let $R$ be a $Z$-type error that affects at most $n$ edges. Constructing the operator $C$ as the Error Correction Algorithm \ref{algoritm} says, then $Syn_{CR}(v) = \mathbf{1}$, by Theorem \ref{multiloop}, $\G(CR)$ is a bridgeless graph, and by Proposition \ref{dificil}, it is a contractible graph. Then by Theorem \ref{multiloop contractil} $CR\ket{\psi} = \ket{\psi}$, and so $C$ corrects the error $R$.
\end{proof}
\begin{example}
Let $G= \Z_2 \times \Z_4$ and consider the lattice in Figure \ref{correcióntotal}. Suppose an error $R$ has occurred on $V_{gs}$. Let us see the correction process as illustrated in Figure \ref{correcióntotal}. In a) are the errors that $R$ has produced. In b), it is the information extracted by the operators $P_{g,\gamma}$. We have labeled the vertices with the elements $Syn_{R}(v)^{-1}$. In c), we have constructed a cluster $N$ that connects the previously labeled vertices. In d), we construct the correction operator $C$, which affects only the edges in $N$, and for every vertex $v$, we have that $Syn_C(v) = Syn_{R}(v)^{-1}$. In e), we overlap the error $R$ and the correction $C$ to see the action of $CR$ on $V_{gs}$. In f), we see how the action of $CR$ is something of form $\prod_{p\in P}\widehat{B}_{\gamma_p}(p)$, then the operator $CR$ acts as the identity on $V_ {gs}$.
\begin{figure}[h]
    \centering
  \includegraphics[scale=1.067]{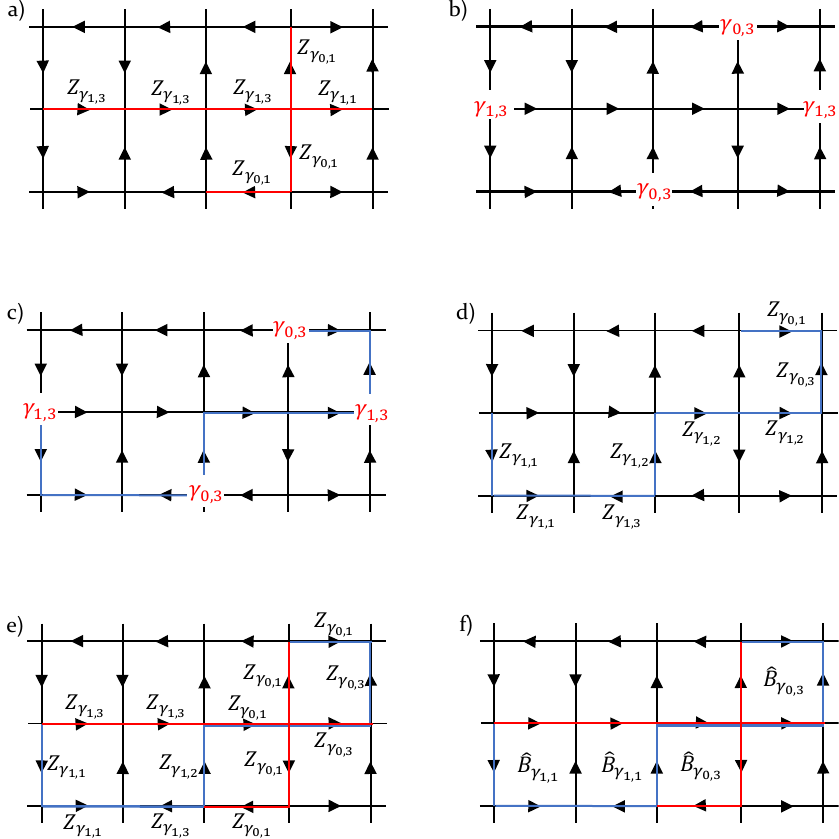}
   \caption{Error correction procedure.}
   \label{correcióntotal}
\end{figure}
\end{example}

\newpage

\section{NP correction problem} \label{apendice}
If we compare Theorems \ref{teoricamente} and \ref{teorema}, we see that to correct $n$ errors, Theorem \ref{teorema} requires that the non-contractible cycles use more than $\left\lfloor n \left( \frac{2+ \log_2(n)}{2} \right) +1 \right\rfloor$ edges, while Theorem \ref{teoricamente} requires that non-contractible cycles use more than $2n$ edges. Then it seems that Theorem \ref{teoricamente} is much better. However, this section will explore some difficulties encountered when constructing an algorithm capable of correcting the set of errors proposed in Theorem \ref{teoricamente}. For this, it suffices to consider $Z$-type errors.

For a $Z$-type error $R$, the first requirement for the correction operator $C$ is that for every vertex $v \in V$, we ask for $Syn_C(v) = Syn_R(v)^{-1}$. With the syndrome of $R$ as the only information, we need to choose the edges on which the correction operator will act. We made this selection of edges with Algorithm \ref{redes}. This algorithm constructs a graph that meets specific conditions we have requested. However, this graph is not necessarily optimal, in the sense that there may exist other graphs that also satisfy the specified conditions but have fewer edges than the graph constructed by Algorithm \ref{redes}. Minimizing the number of edges in the graph $\G(C)$ is crucial because we aim for the edges in the graph $\G(CR)$ to be insufficient for forming a non-contractile cycle and thus guarantee that $\G(CR)$ is contractile, which, according to Theorem \ref{multiloop}, implies that $C$ corrects $R$. Therefore, if there were an algorithm capable of creating graphs with the minimum possible number of edges for a specific syndrome, we could correct as many errors as Theorem \ref{teoricamente} indicates. Yet, as we are going to show in Proposition \ref{optimizararistas}, finding such an optimal graph is an NP-complete problem.

\begin{proposition}\label{optimizararistas}
    Let $R$ be an error operator, for which we only know the syndrome. Finding a graph $\G(C) \subset \mathcal{L}$ with connected components $N_1, \dots, N_m$, each of which is a neutral cluster for the function $x(v) = Syn_R(v)^{-1}$, where $\G(C)$ has the minimum number of edges, is an NP-complete problem.
\end{proposition}

Before presenting the proof of Proposition \ref{optimizararistas}, we will rewrite our
correction graph construction problem in graph-theoretic language. In our problem, we have a set of vertices $W=\{ v_1, \dots, v_r\}$ of a lattice $\cL$. To each vertex $v_k$, we assign an element $Syn_R(v_k)^{-1} \in \widehat{G}$, such that $\prod_{k=1}^rSyn_R(v_k) = \mathbf{1}$. We want to build a graph $\G(C) \subset \cL$ compose of disjoint neutral clusters  $ N_1, \dots, N_m$, such that for every vertex $v_s \in \{ v_1, \dots, v_r\}$, there is a cluster $N_i$ such that $v_s \in N_i$. Let us see that asking for number of edges of the clusters $N_1, \dots,N_m$ to be a minimum turns this problem into an NP problem.

We will do this for the cyclic groups $\Z_m$. Consider a complete graph $\T$ with vertex set $W=\{ v_1, \dots, v_r\}$, and the weight of the edges is distance between the vertices in $\cL$. Thus our problem is expressed as follows: Let $\T = (W, E)$ be the complete graph with vertex set $W$, with a non negative edge weight $w(e)$ for every edge $e$. Furthermore, to each vertex $v$, we assign an element $x(v)$ of the cyclic group $\Z_m$ and suppose that $\sum_{v \in V}x(v) = 0$ in $\Z_m$. We would like to find a collection of connected vertex-disjoint subgraphs of $\T$, $\T_1 = (  W_1, E_1), \dots , \T_k = (W_k, E_k)$, that covers all of the vertices such that $\sum_{v \in W_i}x(v) = 0$ in $\Z_m$ for each $W_i$, and that minimizes $\sum^k_{
i=1}\sum_{e \in E_i}w(e)$.\\

\textbf{Modular Vertex Grouping (MVG($m$))} Given $r, m$, and a graph $\T$ with edge weights and vertex values as above, does there exist an optimal solution with at least two components? 
To do this, consider the following problem.\\

\textbf{Minimum Tree Partition (MTP)} Given a complete graph $\T=(W, E)$ with $|W|=r$, a non negative edge weight $w(e)$ for every edge $e$. The Minimum Tree Partition Problem is to find a partition of $W$ into disjoint sets $\left\{W_i\right\}_{i=1}^k$ such that for all $ i \in \{1, \ldots, k\}$, $\left|W_i\right|=r/k$, and $\sum_{i-1}^mk w\left(\operatorname{MST}\left(W_i\right)\right)$ is minimized, where $\operatorname{MST}(W_i)$ is a minimum spanning tree in the graph induced on $V_i$ and $w\left(E^{\prime}\right)=\sum_{e \in E^{\prime}} w(e)$ for $E^{\prime} \subseteq E$.\\

For the Minimum Tree Partition (MTP) problem with the above considerations, it is evident that $k$ divides $r$. Specifically, in the case of $k = r/2$, the MTP problem simplifies to the Minimum Weight Perfect Matching problem, which can be solved in polynomial time using Edmonds' Blossom Algorithm \cite{blossom}. This simplification also applies to our problem MVG($m$) when $m = 2$, which is the case of the Toric code. Hence, MVG($2$) reduces to the Minimum Weight Perfect Matching problem, allowing the Toric code to correct the number of errors specified in Theorem \ref{teoricamente}. When $m \neq r/2,r$, the MTP problem is a NP-complete \cite{MTP}. General approximation techniques for this problem, as well as other related problems, can be found in \cite{MTP, MTPaprox}. We will use the fact that the MTP is an NP-problem to prove that MVG($m$) with $m > 2$ is also an NP-problem.

\begin{proof}[Proof Proposition \ref{optimizararistas}] 
For the group $G= \mathbb{Z}_m$, the problem of constructing a correction graph is equivalent to the MVG($m$) problem. We will prove the following polynomial-time reduction.
 MTP $\le_p$ MVG($m$): Consider an instance of the Minimum Tree Partition (MTP) problem, with $\G(W,E)$, $|W| = r$, $k < r/2$, and weights $w(e)$. We define a corresponding instance of MVG($m$) as the complete graph on the vertex set $W$, with vertex values $x(v) = 1$ for every $v \in W$, and every edge $e$ is assigned the weight $w(e)$. A set $\{W_i\}_{i=1}^k$ is an optimal solution to the MTP problem if and only if the sum $\sum_{i=1}^k w\left(\operatorname{MST}\left(W_i\right)\right)$ is minimized. In this scenario, the trees $\G_i = \operatorname{MST}\left(W_i\right)$ will form an optimal solution for MVG($r/k$). This is because the set of trees is characterized by the minimization of the weight of the edges within them. Additionally, each tree has $r/k$ vertices, where for each vertex $x(v) = 1$. Therefore, since the group is $Z_{r/k}$, every tree is neutral.
\end{proof}

\begin{remark}
  It is not surprising that the MTP problem, with $|W| = 3r$ and $k = r$ (which corresponds to a case of MVG($3$)), has also been extensively studied. This particular scenario is known as the $P_3$-Partition problem \cite{P3partition}.
\end{remark}

As a consequence of Proposition \ref{optimizararistas}, we arrive at Theorem \ref{propNP}, which states that addressing the set of correctable errors proposed in Theorem \ref{teoricamente} through correction operators, such that the error and its correction together form a contractile graph, involves dealing with an NP-complete problem. Before we examine this result, let us recall the content of Theorem \ref{teoricamente}: it considers $G$ an abelian group, $\Sigma$ a compact oriented surface, and $\mathcal{L}$  an oriented lattice on $\Sigma$, with $2n+1$ being the minimum number of edges in $\mathcal{L}$ required to form a non-contractible cycle in $\Sigma$. Based on this, Theorem \ref{teoricamente} asserts that $V_{gs}$ constitutes a quantum error-correcting code for the set of errors,
\begin{align*}
    \mathfrak{R} = \{ R~|~R \text{ is an error that affects at most } n \text{ edges} \}.
\end{align*}


\begin{theorem} \label{propNP}
Under the hypotheses of Theorem \ref{teoricamente},  constructing a correction operator $C$ for each $R \in \mathfrak{R}$, guaranteeing that $\G(CR)$ is contractile is an NP-complete problem.
\end{theorem}

\begin{proof}
  It is sufficient to consider $Z$-type errors. If an error $R \in \mathfrak{R}_Z$ occurs, the only information we have about $R$ is its syndrome. Assume that $R$ is an error affecting exactly $n$ edges with an optimal graph $\G(R)$ (that is, $\G(R)$ has the minimum number of possible edges for the syndrome of $R$). Any operator $C$ that corrects $R$ must satisfy $Syn_C(v) = Syn_R(v)^{-1}$. To ensure that $\G(CR)$ is contractile, $\G(CR)$ must have at most $2n$ edges, thus $\G(C)$ must have at most $n$ edges, which would imply that $\G(C)$ is an optimal graph for the syndrome $Syn_R(v)^{-1}$. This applies to every error $R$ that affects $n$ edges, with $\G(R)$ being an optimal graph. However, as we know by Proposition \ref{optimizararistas}, finding this optimal graph is an NP-complete problem. Thus, constructing a correction operator $C$ for each $R \in \mathfrak{R}_Z$, guaranteeing that $\G(CR)$ is contractile, is an NP-complete problem.
\end{proof}


\newcommand{\etalchar}[1]{$^{#1}$}

\end{document}